\documentclass[12pt]{amsart}

\usepackage[all,2cell,ps]{xy}
\usepackage{amssymb,amstext,amsmath,amscd,amsthm,amsfonts,enumerate,graphicx,latexsym,color,amsbsy,bm}
\usepackage{todonotes}
\usepackage{hyperref}
\usepackage{tikz-cd}

\newtheorem{thm}{Theorem}[section]
\newtheorem{lem}[thm]{Lemma}
\newtheorem{prop}[thm]{Proposition}
\newtheorem{cor}[thm]{Corollary}

\newtheorem{theorem}[thm]{Theorem}
\newtheorem{lemma}[thm]{Lemma}

\theoremstyle{definition}

\newtheorem{eg}[thm]{Example}

\newtheorem{defn}[thm]{Definition}
\newtheorem{remark}[thm]{Remark}
\newtheorem{rmk}[thm]{Remark}

\theoremstyle{remark}

\newtheorem{example}[thm]{Example}

\usepackage{amsfonts,amsmath,amssymb,amsthm,amscd,amsxtra}
\usepackage{enumerate,epsfig,verbatim}
\usepackage{mathptmx}
\usepackage{amsfonts,amsmath,amssymb,amsthm,amscd,amsxtra}
\usepackage{enumerate,verbatim}
\usepackage{centernot}

\usepackage[all,2cell,ps]{xy}

\usepackage{hyperref}
\usepackage{cleveref}
\usepackage{comment}

\usepackage[all,2cell,ps]{xy}
\usepackage{amsfonts}
\usepackage[margin=1.4in]{geometry}
\usepackage{color}
\usepackage{url}
\usepackage{mathtools}

\usepackage[all,2cell,ps]{xy}
\usepackage{amssymb}
\usepackage{amsmath}
\usepackage{amsfonts}
\usepackage{amsmath}
\usepackage{amsthm}
\usepackage{amssymb}
\usepackage{amscd}
\usepackage{amsfonts}
\usepackage{amsxtra}     
\usepackage{epsfig}
\usepackage{verbatim}
\usepackage[all]{xypic}

\SelectTips{cm}{}

\newcommand{\onto}{\twoheadrightarrow}

\def\latex/{{\protect\LaTeX}}
\def\latexe/{{\protect\LaTeXe}}
\def\amslatex/{{\protect\AmS-\protect\LaTeX}}
\def\tex/{{\protect\TeX}}
\def\amstex/{{\protect\AmS-\protect\TeX}}
\def\bibtex/{{Bib\protect\TeX}}
\def\makeindx/{\textit{MakeIndex}}

\usepackage{amsmath}
\usepackage{amsthm}
\usepackage{amssymb}
\usepackage{amscd}
\usepackage{amsxtra}     
\usepackage{epsfig}
\usepackage{verbatim}
\usepackage[all]{xypic}
\usepackage{enumerate}

\newcommand{\NN}{\mathbb{N}}

\newcommand{\wh}{\widehat}

\DeclareMathOperator{\id}{id}

\DeclareMathOperator{\Tor}{Tor}
\DeclareMathOperator{\Ext}{Ext}
\DeclareMathOperator{\Hom}{Hom}

\DeclareMathOperator{\len}{\ell}
\renewcommand{\len}{\ell}

 \DeclareMathOperator{\Ass}{Ass}
 \DeclareMathOperator{\Soc}{Soc}
 \DeclareMathOperator{\Supp}{Supp}
 \DeclareMathOperator{\Spec}{Spec}

 \DeclareMathOperator{\im}{im}

 \DeclareMathOperator{\pd}{pd}

 \DeclareMathOperator{\depth}{depth}
 \DeclareMathOperator{\coker}{coker}

\def\p{\mathfrak{p}}
\def\fq{\mathfrak{q}}
\def\fm{\mathfrak{m}}
\def\fa{\mathfrak{a}}
\renewcommand{\bar}{\overline}

\DeclareMathOperator{\Fitt}{Fitt}

 \newcommand{\Min}{\textup{Min}}

\newcommand{\Ann}{\textup{Ann}}
\newcommand{\x}{\boldsymbol{x}}

\def\urltilda{\kern -.15em\lower .7ex\hbox{\~{}}\kern .04em}\def\urldot{\kern -.10em.\kern -.10em}\def\urlhttp{http\kern -.10em\lower -.1ex\hbox{:}\kern -.12em\lower 0ex\hbox{/}\kern -.18em\lower 0ex\hbox{/}}

\def\Tr{\mathrm{Tr}}

\DeclareMathOperator{\e}{\operatorname{{e}}}
\DeclareMathOperator{\E}{\operatorname{{E}}}

\DeclareMathOperator{\rank}{\operatorname{{rank}}}
\DeclareMathOperator{\drank}{\operatorname{{d-rank}}}

\numberwithin{equation}{section}

\usepackage{amscd,amssymb,amsopn,amsmath,amsthm,mathrsfs,graphics,amsfonts,enumerate,verbatim,calc
}

\usepackage{tikz}
\usepackage{lscape}
\usetikzlibrary{matrix,arrows,decorations.pathmorphing}

\usetikzlibrary {positioning}
\usetikzlibrary{patterns}
\usetikzlibrary{calc}
\definecolor {processblue}{cmyk}{0.96,0,0,0}

\usepackage{color}

\DeclareFontFamily{OT1}{rsfs}{}
\DeclareFontShape{OT1}{rsfs}{n}{it}{<-> rsfs10}{}
\DeclareMathAlphabet{\mathscr}{OT1}{rsfs}{n}{it}

\numberwithin{equation}{section}
\newcommand{\codim}{\operatorname{codim}}

\newcommand{\length}{\ell}
\renewcommand{\length}{\ell}

\newcommand{\al}{\alpha}

\newcommand{\m}{\mathfrak{m}}

\newcommand{\fp}{\mathfrak{p}}

\renewcommand{\le}{\leqslant}\renewcommand{\leq}{\leqslant}
\renewcommand{\ge}{\geqslant}\renewcommand{\geq}{\geqslant}

\DeclareMathOperator{\Assh}{Assh}

\DeclareMathOperator{\torsion}{tor}

\DeclareMathOperator{\sgn}{sgn}

\newcommand{\w}{\omega}

\begin{document}

\title[On a generalization of Ulrich modules and its Applications]{On a generalization of Ulrich modules and its Applications}

\author[E.\ Celikbas]{Ela Celikbas}
\address{Ela Celikbas\\
School of Mathematical and Data Sciences\\
West Virginia University\\
Morgantown, WV 26506-6310, U.S.A}
\email{ela.celikbas@math.wvu.edu}

\author[O.\ Celikbas]{Olgur Celikbas}
\address{Olgur Celikbas\\
School of Mathematical and Data Sciences\\
West Virginia University\\
Morgantown, WV 26506-6310, U.S.A}
\email{olgur.celikbas@math.wvu.edu}

\author[J.\ Lyle]{Justin Lyle}
\address{Justin Lyle \\
Department of Mathematical Sciences\\
University of Arkansas\\
Fayetteville, AR 72701 USA}
\email{jl106@uark.edu}
\urladdr{https://jlyle42.github.io/justinlyle/}

\author[R.\ Takahashi]{Ryo Takahashi}
\address{Ryo Takahashi\\
Graduate School of Mathematics, Nagoya University, Furocho, Chikusaku, Nagoya, Aichi 464-8602, Japan}
\email{takahashi@math.nagoya-u.ac.jp}
\urladdr{https://www.math.nagoya-u.ac.jp/~takahashi/}

\author[Y.\ Yao]{Yongwei Yao}
\address{Yongwei Yao\\
Department of Mathematics and Statistics, Georgia State University, Atlanta, Georgia 30303, U.S.A.}
\email{yyao@gsu.edu}

\date{\today}



\begin{abstract}
We study a modified version of the classical Ulrich modules, which we call $c$-Ulrich. Unlike the traditional setting, $c$-Ulrich modules always exist. We prove that these modules retain many of the essential properties and applications observed in the literature. Additionally, we reveal their significance as obstructions to Cohen-Macaulay properties of tensor products. Leveraging this insight, we show the utility of these modules in testing the finiteness of homological dimensions across various scenarios.
\end{abstract}

\maketitle

\section{Introduction}

Let $R$ be a commutative Noetherian local ring. Over the past several decades, there has been a great deal of activity in studying maximal Cohen-Macaulay $R$-modules requiring a large number of generators \cite{BH89,BH87,CM15,CM12,DG23,DK23,GO16,HH05,Ha99,HU91,JL07,KT19,LM19,Ma20,Ul84,Yh21}. 
Those that require the largest possible number are known in the literature as Ulrich modules, and they bear deep connections to several distinct areas of commutative algebra and algebraic geometry. They are naturally connected with multiplicity theory and measures of singularity, they have test module properties in the homological sense, and their existence is known to imply a famous and long-standing conjecture of Lech on flat local maps \cite{HH05,Ha99,JL07,LM19,Ul84}. Indeed, work of Ma on sequences of modules that become Ulrich in a limiting sense was recently used to establish Lech's conjecture in the graded setting \cite{Ma20}. On the other hand, results of Yhee show that Ulrich modules in the traditional sense need not exist, and even the milder sequences used by Ma in the graded setting do not exist in general \cite{Yh21}. So, it is natural to explore weakenings of the Ulrich condition that are more abundant, but which still retain some of the desired properties and applications of Ulrich modules. 

In this paper, we consider several angles along which the Ulrich condition may be weakened. By relaxing the maximal Cohen-Macaulay requirement, we ensure the notions we consider will always exist, and we show that many of the familiar properties of Ulrich modules are retained in our setting. Our novel approach is to tie these generalized notions to Cohen-Macaulay properties of tensor products, itself an active line of investigation (see e.g. \cite{CG19,Co95,HW94}), and in doing so, we are able to refine, expand, and strengthen numerous results about Ulrich modules from the literature. Along the way, we provide several technical tools of independent interest. 

To be concrete, given a real number $c$, we say a Cohen-Macaulay $R$-module $M$ is \emph{$c$-Ulrich} if $\e_R(M) \leq c \cdot \mu_R(M)$. If $c=1$, we simply call $M$ an Ulrich module. The classically studied notion of an Ulrich module discussed previously corresponds to the case where $c=1$ and $M$ is maximal Cohen-Macaulay.
The following is a key special case of our first main theorem. 

\begin{thm}\label{intro:mainthm} Let $R$ be a Cohen-Macaulay local ring and let $M$ and $N$ be nonzero finitely generated $R$-modules. Suppose that $N$ has positive rank $r$ and that $M$ is $c$-Ulrich for some $c<1 + \frac{1}{r}$. If $M\otimes_RN$ is maximal Cohen-Macaulay, then $N$ is free.
\end{thm}
Under the stronger assumption that $M$ is Ulrich, we obtain the same conclusion under more flexible hypotheses.
\begin{thm} \label{mainthmintro} Let $R$ be a local ring and let $M$ and $N$ be finitely generated $R$-modules with $M \ne 0$. Suppose $M$ is Ulrich and that at least one of the following holds:
\begin{enumerate}[\rm(i)]
\item $M$ is faithful.
\item $N$ has rank.
\item $\Tor_1^R(M,N)=0$.
\end{enumerate}
If $M \otimes_R N$ is maximal Cohen-Macaulay, then $N$ is free.
\end{thm}
Huneke-Wiegand proved that, if $R$ is a hypersurface ring and $M$ and $N$ are finitely generated $R$-modules such that $M$ or $N$ has rank and $M\otimes_RN$ is maximal Cohen-Macaulay, then $M$ or $N$ is free \cite[3.1]{HW94}. Hence Theorem \ref{mainthmintro} shows that the homological behavior of Ulrich modules somewhat resembles that of modules over hypersurface rings.

Tensor products of nonzero modules tend to have torsion, but it is not always the case. There are many notable examples of non-free (or even non maximal Cohen-Macaulay) modules in the literature whose tensor product is maximal Cohen-Macaulay. Therefore, the assumption in Theorems \ref{intro:mainthm} and \ref{mainthmintro} that $M\otimes_R N$ is maximal Cohen-Macaulay does not necessarily imply the freeness of $N$ unless the Ulrich or $c$-Ulrich hypothesis on $M$ is satisfied. In Examples \ref{egT2} and \ref{egT3}, we present such instances where both $M$ and $N$ have rank to illustrate this scenario.

There are several notable consequences of our main theorems, including expansions of several results on test module and rigidity properties from the literature (Corollaries \ref{pdtest:alpha}, \ref{corExt1}, and \ref{corExt2}), and applications (Corollary \ref{hunwie}) to the famous and long-standing Huneke-Wiegand conjecture \cite{HW94}. We also use a certain numerical deviation from the Ulrich condition to provide a lower bound on the growth rate of Betti numbers, providing a variation on a result of Jorgensen-Leuschke and recovering a well-known result of Avramov in the process (see Theorem \ref{Ulrichgrowtoseh} and Corollary \ref{canonicalUlrich}). Some of the key applications include the following; see Corollaries \ref{corExt2} and \ref{hunwie} for more details:

\begin{cor} \label{corRintro} Let $R$ be a $d$-dimensional Cohen-Macaulay local domain and let $M$ and $N$ be nonzero finitely generated $R$-modules. Suppose $N$ is maximal Cohen-Macaulay with $r=\rank(N)$.

\begin{enumerate}[\rm(i)]

\item If $M$ is $c$-Ulrich for some $c<1+\frac{1}{r}$ and $\Ext^i_R(M,N)=0$ for $i=1,\dots,d$, then $\id_R(N)<\infty$. 

\item If $M$ is Ulrich, $d=1$, and $M \otimes_R M^*$ is torsion-free, then $M$ is free and $R$ is regular. In particular, the Huneke-Wiegand conjecture holds true for Ulrich modules. 

\end{enumerate}
\end{cor}

During the preparation of this manuscript, the conclusion of Corollary 1.3(ii) was also obtained independently by Dey and Kobayashi via different techniques; see \cite[5.2 and 7.4(1)]{DK23}.

Finally, we show that if $R$ is Cohen-Macaulay of codimension $v$ and $M$ is $c$-Ulrich for some $c<1+\frac{1}{v}$, then, after extending the residue field $k$ and cutting down by a reduction of the maximal ideal on $M$, the resulting module has a copy of $k$ as a direct summand (see Corollary \ref{cUlrichtest}). Thus $c$-Ulrich modules for such a value of $c$ carry an exceedingly strong form of rigidity, and can be used to test any homological dimension that can be detected by $k$.

\section{Preliminaries}

Throughout, $(R, \fm, k)$ denotes a commutative Noetherian local ring, and $I$ denotes an $\fm$-primary ideal of $R$. All modules over $R$ are assumed to be finitely generated.

\begin{defn} \label{multprop} Given an $R$-module $M\neq 0$ with $\dim_R(M)=d$, we set $$\e_R(I,M)= d! \lim_{n \to \infty} \dfrac{\length_R(M/I^n M)}{n^{d}},$$ to be the \emph{Hilbert-Samuel multiplicity} of $M$ with respect to $I$. We set $\e_R(I,0)$ to be $0$, and we write $\e_R(M)=\e_R(\m,M)$. 
\end{defn}

According to our definition, $\e_R(I, -)$ is not necessarily additive on arbitrary short exact sequences of $R$-modules. However if $0 \to X \to Y \to Z \to 0$ is an exact sequence of $R$-modules such that $X$, $Y$ and $Z$ all have the same dimension, then $\e_R(I,Y)=\e_R(I, X)+\e_R(I, Z)$. The advantage of our definition is that $\e_R(I,M)=0$ if and only if $M=0$ (see \cite[Section 4]{BH93}).

If $M$ is an $R$-module, we write $\mu_R(M)$ for the minimal number of generators of $M$ and $\length_R(M)$ for its length. We write $\Omega^i_R(M)$ for the $i$th syzygy of $M$. The following is well known; the standard proof is a mild adaptation of that of \cite[1.1]{BH87}:

\begin{rmk} \label{bound} Let $M$ be a Cohen-Macaulay $R$-module. Then $\length_R(M/IM) \leq \e_R(I,M)$. So we have
$$\mu_R(M)\leq \length_R(M/IM) \leq \e_R(I,M) \leq \mu_R(M) \cdot \e_R\big(I,R/\Ann_R(M)\big).$$ In particular $\mu_R(M)$ cannot exceed $\e_R(M)$.
\end{rmk}

The following definition gives the primary notion we will study.

\begin{defn}\label{ulrichdefn} Let $M$ be a Cohen-Macaulay $R$-module and let $c$ be a real number. We say $M$ is \textit{c-Ulrich with respect to $I$} if the following conditions hold:
\begin{enumerate}[\ \ \rm(a)]
\item $M/IM$ is free over $R/I$, 
\item $\e_R(I,M) \le c \cdot \length_R(M/IM)$.
\end{enumerate}

If $I=\fm$, we simply say that $M$ is \textit{$c$-Ulrich}. Similarly, if $M$ is $1$-Ulrich with respect to $I$, we say that $M$ is an \textit{Ulrich module with respect to $I$}. We call $M$ \textit{Ulrich} if it is $1$-Ulrich with respect to $\fm$. 
\end{defn}

Although Ulrich modules are generally assumed to be maximal Cohen-Macaulay in the literature, our definition only requires an Ulrich module to be Cohen-Macaulay. In other words, according to our definition and in view of Remark \ref{bound}, a module is \emph{Ulrich} if and only if it is a maximally generated Cohen-Macaulay module. We will show in later sections that many of the familiar properties of Ulrich modules in the classical sense are retained by our notion. Moreover, as the next example indicates, Ulrich modules in our sense always exist, whereas maximal Cohen-Macaulay Ulrich modules do not exist in general \cite{Yh21}. 

\begin{eg}\label{ulrichex} \mbox{}
\begin{enumerate}[\rm(i)]
\item If $I$ is an $\fm$-primary ideal of $R$, then $R/I$ is an Ulrich $R$-module with respect to $I$.
\item We have $k^{\oplus n}$ is an Ulrich $R$-module for each $n\geq 1$. Moreover, if $M$ is an Ulrich module of finite length, then $M \cong k^{\oplus n}$ for some $n\geq 1$; see \cite[1.3]{BH87}. 
\item If $R$ is Cohen-Macaulay with $\dim R=1$, then $\m^n$ is an Ulrich module for all $n \ge \e(R)-1$; see \cite[2.5]{SV74}. If additionally $R$ is analytically unramified, then the integral closure $\bar{R}$ over $R$ is also Ulrich. Indeed, $\bar{R}$ is a product of DVRs, so $\m\bar{R}=x\bar{R}$ for some nonzerodivisor $x \in \m$. 
\item If $R$ is Cohen-Macaulay and has minimal multiplicity, that is, if $\e(R)=\mu_R(\fm)-\dim(R)+1$, and $M$ is maximal Cohen-Macaulay, then $\Omega_R^i M$ is Ulrich for each $i\geq 1$; see \cite{BH87}.
\item If $R$ is a Cohen-Macaulay local ring with a canonical module $\omega_R$, then $\omega_R$ is Ulrich if and only if $R$ is regular. Indeed, we may suppose $k$ is infinite and take a minimal reduction $\x$ of $\m$. Then $\w_R$ being Ulrich forces $\mu_R(\w_R)=\dim_k(\Soc(R/\x R))=\length_R(R/\x R)=\e(R)$ so that $R/\x R$ is a field, which implies $R$ is regular.
\item In view of Remark \ref{bound}, an $R$-module $M$ is $\e_R(R/\Ann_R(M))$-Ulrich if and only if it is Cohen-Macaulay.
\end{enumerate}
\end{eg}

In addition to the cases mentioned in Example \ref{ulrichex}, there are other classes of rings that are known to admit maximal Cohen-Macaulay Ulrich modules. For example, strict complete intersection rings, two-dimensional graded domains over an infinite field, certain Veronese algebras, and coordinate rings of  Grassmannian, Segre, and certain determinantal varieties; see \cite{BH87,CM15,CM12,ES03,Ha05,HU91}. While this list is not comprehensive, is it not far from being so; there are few other cases where maximal Cohen-Macaulay $R$-modules are known to exist.

We present several additional examples of Ulrich and $c$-Ulrich modules, many of which will recur for specific purposes in later sections. Owing to their concrete nature, numerical semigroup rings provide abundant examples of Ulrich and $c$-Ulrich modules.

\begin{eg} Let $R=k[\![t^{a_1}, \ldots, t^{a_n}]\!]$ be a numerical semigroup ring with $0<a_1<\ldots<a_n$ and $\gcd(a_1, \ldots, a_n)=1$.

\begin{enumerate}[\rm(i)]

\item $R$ is $a_1$-Ulrich $R$-module since $\e(R)=a_1$ and $\mu_R(R)=1$. 
\item $\bar{R} \cong k[\![t]\!]$ is an Ulrich module (see Example \ref{ulrichex} (iii)).
\item $(t^{a_1},\dots,t^{a_n})$ is an $(a_1/n)$-Ulrich module. 
\item If $n=3$ and $R$ is not Gorenstein, then the canonical module $\w_R$ of $R$ is $(a_1/2)$-Ulrich and $\Omega^1_R(\w_R)$ is $(a_1/3)$-Ulrich. Indeed, it follows from \cite{He70} that $\mu_R(\w_R)=2$ and $\beta^R_1(\w_R)=3$ and additivity of multiplicity gives $\e_R(\w_R)=\e_R(\Omega^1_R(\w_R))=a_1$.

\end{enumerate}

\end{eg}

The following example demonstrates that maximal Cohen-Macaulay Ulrich modules induce examples of Ulrich modules that are not maximal Cohen-Macaulay, and so it is natural in this sense to work outside the maximal Cohen-Macaulay setting.

\begin{eg} Let $R$ be a local ring. Assume $M$ is a maximal Cohen-Macaulay Ulrich $R$-module. Set $S=R[\![x_1, \ldots, x_n]\!]$ for some variables $x_1, \ldots, x_n$. Then $M$ has an $S$-module structure via the homomorphism $S\to R$ where $x_i \mapsto 0$. Note that $M$, as an $S$-module, is Ulrich, but it is not maximal Cohen-Macaulay. 
\end{eg}

\begin{eg} Let $R=k[\![x]\!]/(x^{n})$ with $n \ge 1$. Then $R/(x^i)$ is $i$-Ulrich for each $i=1,\dots,n$. 
\end{eg}

\begin{eg} Let $R=k[\![x,y]\!]/(x^3+y^4)$ and set $M=\m$. Then $R$ is a one-dimensional hypersurface, and $M$ is a maximal Cohen-Macaulay module. Moreover, $\e_R(M)=\e(R)=3$ and $\mu_R(M)=2$. Hence $M$ is $3/2$-Ulrich. 
\end{eg}

\begin{eg} Let $R=k[\![x,y]\!]/(x^2,xy,y^2)$. Then the injective hull $\E_R(k)$ of $k$ is $(3/2)$-Ulrich. This follows since $\e_R(\E_R(k))=\e(R)=\length_R(R)=3$ and $\mu_R(\E_R(k))=2$ (as the type of $R$ is $2$). 
\end{eg}

\begin{eg} \label{ex:4/3} Let $R=k[\![x,y,z]\!]/(xy,z^2)$ and set $M=\m$. Then $R$ is a one-dimensional complete intersection of codimension two, and $M$ is a maximal Cohen-Macaulay module. Moreover, $\e_R(M)=\e(R)=4$ and $\mu_R(M)=3$. Hence $M$ is $(4/3)$-Ulrich.
\end{eg}

\begin{eg} \label{ex:3/2} Let $R=k[\![t^3, t^4, t^5]\!]$. Then the canonical module $\omega_R=(t^3, t^4)$  is $(3/2)$-Ulrich.
\end{eg}

The next remark, while elementary, is indicative of the need for the condition in Definition \ref{ulrichdefn} that $M/IM$ be a free $R/I$ module, and will be used repeatedly.

\begin{rmk} \label{prelim} Let $M$ and $N$ be $R$-modules such that $M/IM$ and $N/IN$ are both free over $R/I$. 
\begin{enumerate}[\rm(i)]
\item As $M/IM$ is free over $R/I$, it follows that $M \cong (R/I)^{\oplus n}$ for some integer $n$. This implies that $M/\fm M \cong k^{ \oplus n}$, which implies that $\mu_R(M) = n$. Therefore $M/IM$ is a free $R/I$-module of rank $\mu_R(M)$. Similarly $N/IN$ is a free $R/I$-module of rank $\mu_R(N)$. Consequently we conclude that $(M \otimes_R N) /I(M \otimes_R N)$ is a free module over $R/I$ of rank $\mu_R(M \otimes_R N) = \mu_R(M) \cdot \mu_R(N)$. 
\item It follows from part (i) that $\len_R(N/IN)=\mu_R(N) \cdot \len_R(R/I)$, $\len_R(M/IM)=\mu_R(M) \cdot \len_R(R/I)$, and $\len_R\big((M\otimes_RN) /I (M\otimes_RN)\big)=\len_R(R/I) \cdot \mu_R(M) \cdot \mu_R(N) = \len_R(R/I) \cdot \mu_R(M \otimes_R N)$.
\end{enumerate}
\end{rmk}

\begin{defn} \label{maindefn} Let $M$ be an $R$-module.
\begin{enumerate}[\rm(a)]
\item We say $M$ is free on $Y\subseteq \Spec(R)$ if $M_{\fp}$ is a free $R_{\fp}$-module for all $\fp \in Y$. 
\item We say $M$ has rank $r$ on $Y\subseteq \Spec(R)$ if $M_{\fp}\cong R^{\oplus r}_{\fp}$ for all $\fp \in Y$. 
\item If $M$ has rank $r$ on $\Ass(R)$, then we may simply say that $M$ has \emph{rank} $r$.
\item We set $\Assh_R(M)=\{\fp \in \Ass_R(M) : \dim(R/\fp)=\dim_R(M)\}$.
\end{enumerate}
\end{defn}
\begin{defn} \label{unmixed} Given an $R$-module $M$, we call $M$ \emph{unmixed} if $\Ass_R(M)=\Assh_R(M)$. \footnote{We note that our definition of unmixed differs slightly from some notions in the literature, in that our definition does involve completion.}
\end{defn}
It follows by definition that $\Assh_R(M) \subseteq \Min_R(M) \subseteq \Ass_R(M)$. If $M$ is unmixed and $0\neq N$ is an $R$-submodule of $M$, then $\dim(R/\fq)=\dim_R(M)$ for each $\fq \in \Ass_R(N)$ so that $\dim_R(N)=\dim_R(M)$.
\begin{lem} \label{lemrk} Let $R$ be a commutative ring, $M$ be an $R$-module, and let $\emptyset \neq Y \subseteq \Spec(R)$.
\begin{enumerate}[\rm(i)]
\item If $M$ is free on $Y$ and $M$ has rank $r$, then $M$ has rank $r$ on $Y$.
\item If $M$ has rank $r$ on $Y$ and $\Supp_R(M) \cap Y \neq \emptyset$, then $r>0$.
\end{enumerate}
\end{lem}

\begin{proof} (i) Assume $M$ is free on $Y$ and $M$ has rank $r$. Let $\fq \in Y$. Then $M_{\fq} \cong R_{\fq}^{\oplus v}$ for some $v\geq 0$, and there exists $\fp \in \Min(R)$ such that $\fp \subseteq \fq$. So $R_{\fp}^{\oplus v}\cong (R^{\oplus v}_{\fq})_{\fp R_{\fq}} \cong (M_{\fq})_{\fp R_{\fq}} \cong M_{\fp} \cong R_{\fp}^{\oplus r}$. This shows that $r=v$ and so $M$ has rank $r$ on $Y$. 

(ii) Assume $M$ has rank $r$ on $Y$ and $\Supp_R(M) \cap Y \neq \emptyset$. Then there exists $\fq \in \Supp_R(M) \cap Y$ so that $0\neq M_{\fq} \cong R_{\fq}^{\oplus r}$. Therefore $r>0$.
\end{proof}

\begin{lem} \label{mainlem} Let $M$ and $N$ be $R$-modules.
\begin{enumerate}[\rm(i)]
\item If $\dim_R(M)=\dim_R(M\otimes_RN)$, then $\Assh_R(M\otimes_RN) =\Assh_R(M) \cap \Supp_R(N) $. 
\item Assume $\dim_R(M)=\dim_R(M\otimes_RN)$. If $N$ has rank $r$ and $N$ is free on $\Assh_R(M\otimes_RN)$, then $N$ has rank $r$ on $\Assh_R(M\otimes_RN)$ and $r>0$.
\item If $\Supp_R(N)=\Spec(R)$ (e.g., $N$ has positive rank), then $\Assh_R(M\otimes_RN) =\Assh_R(M)$.
\end{enumerate}
\end{lem}

\begin{proof} (i) For $\fp \in \Assh_R(M\otimes_RN)$, we see $\fp \in \Supp_R(M\otimes_RN)\subseteq \Supp_R(M)$ and $\dim(R/\fp)=\dim_R(M)$. Therefore $\fp \in \Min_R(M) \subseteq \Ass_R(M)$ and hence $\fp \in \Assh_R(M)$. This shows the inclusion that $\Assh_R(M\otimes_RN) \subseteq \Assh_R(M) \cap \Supp_R(N)$. On the other hand, if $\fq \in \Assh_R(M) \cap \Supp_R(N)$, then $\fq \in \Supp_R(M\otimes_RN)$ and $\dim(R/\fq)=\dim_R(M)=\dim_R(M\otimes_RN)$ so that $\fq \in \Assh_R(M\otimes_RN)$. Consequently, $\Assh_R(M\otimes_RN) =\Assh_R(M) \cap \Supp_R(N)$. 

(ii) The claims follow by part (i) and Lemma \ref{lemrk}.

(iii) This follows from part (i) immediately.
\end{proof}

\section{Main theorems}

In this section, we prove our main theorems. We begin with some preparatory results that will be instrumental in the proofs.

\begin{lem}\label{lem:J}
Let $R$ be a local ring and let $M$ be an $R$-module. Set $J = \Ann_R(M)$. For any $R$-module $N$, the following statements are equivalent: 
\begin{enumerate}[\rm(i)]
\item $M^{\oplus \mu_R(N)} \cong M \otimes_R N$.
\item All entries of the minimal presentation matrix of $N$ are in $J$ (that is, $\Fitt_{\mu_R(N)-1}(N) \subseteq J$).
\item $L^{\oplus \mu_R(N)} \cong L \otimes_R N$ for all $R$-modules $L$ such that $\Ann_R(L) \supseteq J$.
\item $(R/J)^{\oplus \mu_R(N)} \cong N/JN$.
\end{enumerate}
\end{lem}
 
\begin{proof}

Let $N$ be a $R$-module with a minimal presentation $R^{\oplus b_1} \xrightarrow{A} R^{\oplus \mu_R(N)} \to N \to 0$. Then
\begin{align*}
M^{\oplus \mu_R(N)} \cong M \otimes_R N & \implies M^{\oplus b_1} \xrightarrow{A} M^{\oplus \mu_R(N)} \xrightarrow{\cong} M\otimes_R N \to 0 \text{ is exact}\\
& \implies A \equiv 0 \mod J \\
& \implies L^{\oplus b_1} \xrightarrow{0} L^{\oplus \mu_R(N)} \xrightarrow{} L\otimes_R N \to 0 \text{ is exact for all $L$ with $\Ann(L) \supseteq J$}\\
& \implies L^{\oplus \mu_R(N)} \cong L \otimes_R N \text{ for all $L$ with $\Ann(L) \supseteq J$} \\
& \implies (R/J)^{\oplus \mu_R(N)} \cong (R/J) \otimes_R N \cong N/JN \\
& \implies M^{\oplus \mu_R(N)} \cong M \otimes_R (R/J)^{\oplus \mu_R(N)} \cong M \otimes_R (R/J) \otimes_R N \cong M \otimes_R N. \qedhere
\end{align*}
\end{proof}
 
\begin{prop}\label{High Bridge Lemma} Let $R$ be a local ring and let $M \neq 0$, $N$ be $R$-modules such that $M \otimes_R N \cong M^{\oplus \mu_R(N)}$. Set $J = \Ann_R(M)$. Then
\begin{enumerate}[\rm(i)]
\item If $M$ is faithful, then $N$ is free over $R$.
\item $\Tor_1^R(L, N) \cong L \otimes_R \Omega^1_R(N) \cong L \otimes_R \Tor_1^R(R/J, N)$ for all $R$-modules $L$ with $\Ann_R(L) \supseteq J$.
\item $N$ is free over $R$ if and only if $\Tor_1^R(L, N) = 0$ for some $R$-module $L \neq 0$ with $\Ann_R(L) \supseteq J$ (for example, $\Tor_1^R(M, N) = 0$ or $\Tor_1^R(R/J, N) = 0$).
\item Assume $\Supp(M) = \Spec(R)$. Then $N$ is free over $R$ if and only if $\Tor_1^R(M,N)$ is torsion.
\end{enumerate}
\end{prop}
 
\begin{proof}
(i) By Lemma~\ref{lem:J}, $N/JN$ is free over $R/J$. If $J=0$, then $N$ is free over $R$.
 
(ii) Let $L$ be any $R$-module with $\Ann(L) \supseteq J$. Applying $L \otimes_R -$ to the short exact sequence $0 \to \Omega^1_R(N) \to R^{\oplus \mu_R(N)} \to N \to 0$ and observing Lemma~\ref{lem:J}, we get an exact sequence
\[
0 \to \Tor_1^R(L, N) \to L \otimes_R \Omega^1_R(N) \to L^{\oplus \mu_R(N)} \xrightarrow{\cong} L \otimes_R N \to 0,
\]
which implies $\Tor_1^R(L, N) \cong L \otimes_R \Omega^1_R(N)$, which then yields
\[\Tor_1^R(L, N) \cong L \otimes_R \Omega^1_R(N) \cong L \otimes_R [(R/J) \otimes_R \Omega^1_R(N)] \cong L \otimes_R \Tor_1^R(R/J, N).
\]
 
(iii) If $N$ is free over $R$, then $\Tor_1^R(L, N) = 0$ for all $R$-module $L$. Conversely, if $\Tor_1^R(L, N) = 0$ for some $R$-module $L \neq 0$ with $\Ann_R(L) \supseteq J$, then $L \otimes_R \Omega^1_R(N) = 0$ by part (2) above, which  implies $\Omega^1_R(N) = 0$, so $N \cong R^{\oplus \mu_R(N)}$.

(iv) If $N$ is free, then clearly $\Tor_1^R(M,N)$ is torsion. Conversely, assume $\Tor_1^R(M,N)$ is torsion. By (ii) above, $M \otimes_R \Omega^1_R(N)$ is torsion, meaning that $M_\fp \otimes_R \Omega^1_R(N)_\fp = 0$ for all $\fp \in \Ass(R)$. In light of $\Supp(M) = \Spec(R)$, we conclude $\Omega^1_R(N)_\fp = 0$ for all $\fp \in \Ass(R)$. This implies $\Omega^1_R(N) = 0$.
\end{proof}

\begin{lem} \label{lem:c} Let $M \neq 0$ and $N$ be $R$-modules. Assume that $M/IM$ and $N/IN$ are both free over $R/I$, $\dim_R(M)= \dim_R(M \otimes_R N)$, $\e_R(I,M) \le c \cdot \length_R(M/IM)$, and $M\otimes_R N$ is Cohen-Macaulay. Let $J = \Ann_R(M)$. Then the following hold:
\begin{enumerate}[\rm(i)]
\item  The module $M \otimes_R N$ is $c$-Ulrich with respect to $I$ such that
$$\e_R(I,M \otimes_R N) \leq  \e_R(I,M) \cdot \mu_R(N) \le c \cdot \length_R\big((M \otimes_R N)/ I (M \otimes_R N)\big)  \leq c \cdot \e_R(I,M \otimes_R N).$$  In particular, if $\e_R(I,M)= \length_R(M/IM)$, then $$\e_R(I,M \otimes_R N) =  \e_R(I,M) \cdot \mu_R(N)=\length_R\big((M \otimes_R N)/ I (M \otimes_R N)\big),$$
so that $M \otimes_R N$ is Ulrich with respect to $I$.
\item Assume that either $M$ or $N/JN$ has rank over $R/J$ on $\Assh_{R/J}(R/J)$. Then we have the inequality $\mu_R(N) \le c \cdot \dfrac{\e_R(I,\, N/JN)}{\e_R(I,\, R/J)}$. In particular, if $\e_R(I,M)= \length_R(M/IM)$, then we have $\mu_R(N) = \dfrac{\e_R(I,\, N/JN)}{\e_R(I,\, R/J)}$.
\end{enumerate}
\end{lem}

\begin{proof} (i) There is a surjection $R^{\oplus \mu_R(N)} \twoheadrightarrow N$. We have the following (in)equalities:
\[
\begin{array}{rclclc} 
\e_R(I,M \otimes_R N) &\le& \displaystyle \e_R(I,M^{\oplus \mu_R(N)}) && \\
&=& \displaystyle \e_R(I,M) \cdot \mu_R(N) && \\
&\le& \displaystyle c \cdot \len_R(M/IM)  \cdot \mu_R(N) && \\
&=& \displaystyle c \cdot \mu_R(M) \cdot \len_R(R/I)  \cdot \mu_R(N) && \\
&=& \displaystyle c \cdot \mu_R(M \otimes_R N) \cdot \len_R(R/I) && \\
&=& \displaystyle c \cdot \len_R\big((M \otimes_R N)/ I (M \otimes_R N)\big) &\le& \displaystyle  c \cdot \e_R(I,M \otimes_R N). \\
\end{array}\]
In the above, the first and the last inequalities follow from the surjection $M^{\oplus \mu_R(N)} \onto M \otimes_R N$ and  Remark~\ref{bound} respectively, while the second, the third and the fourth equalities are covered in Remark~\ref{prelim}(ii). In summary, we have
$$\e_R(I,M \otimes_R N) \leq  \e_R(I,M) \cdot \mu_R(N) \le c \cdot \length_R\big((M \otimes_R N)/ I (M \otimes_R N)\big)  \leq c \cdot \e_R(I,M \otimes_R N).$$
Therefore $M \otimes_R N$ is $c$-Ulrich with respect to $I$. When $c = 1$, the following equalities hold:
\[
\e_R(I,M \otimes_R N) =  \e_R(I,M) \cdot \mu_R(N)=\length_R\big((M \otimes_R N)/ I (M \otimes_R N)\big). 
\]

(ii) By (i) above, we see $\mu_R(N) \le c \cdot \dfrac{\e_R(I,\,M \otimes_R N)}{\e_R(I,\,M)}$. Note that $M \otimes_R N \cong M \otimes_R (N/JN)$. Since either $M$ or $N/JN$ has rank over $R/J$, we use the associativity formula \cite[4.7.8]{BH93} to obtain 
\[\e_R(I,M \otimes_R N)=\e_R(I,M \otimes_R (N/JN))=\dfrac{\e_R(I,M)\e_R(I,N/JN)}{\e_R(I,R/J)}.\] Consequently, we have that $\dfrac{\e_R(I,\,M \otimes_R N)}{\e_R(I,\,M)}=\dfrac{\e_R(I,N/JN)}{\e_R(I,R/J)}$ and the claim follows.
\end{proof}

We are now ready to prove our first main theorem which encompasses Theorem \ref{intro:mainthm}. 

\begin{thm} \label{mainthm:alpha} Let $M \neq 0$ and $N$ be $R$-modules. Assume that $M/IM$ and $N/IN$ are both free over $R/I$, $\dim_R(M)=\dim_R(M\otimes_RN)$, $M\otimes_RN$ is Cohen-Macaulay, and $\e_R(I,M) \le c \cdot \length_R(M/IM)$ for some real number $c$ (e.g., $M$ is $c$-Ulrich with respect to $I$). Then $N$ is free over $R$ if at least one of the following holds:
\begin{enumerate}[\rm(i)]
\item $N$ has rank, $\Tor_1^R(M,N)_{\fp}=0$ for \emph{some} $\fp \in \Assh_R(M)$, and $c < 1+ \dfrac{1}{\e_R(I,\,M \otimes_R N)}$.
\item $N$ is free on $\Assh(M)$, $N$ has rank $r$ (which must be positive), and $c < 1+ \dfrac {1}{r}$. 
\item $M$ is unmixed, $c < 1+ \dfrac{1}{\e_R(I,\,M \otimes_R N)}$, and at least one of the following holds:
\begin{enumerate}[\rm(a)]
\item $M$ is faithful.
\item $\Tor_1^R(M,N)=0$.
\item $\Tor_1^R(M,N)$ is torsion and $\Supp_R(M)=\Spec(R)$.
\end{enumerate}
\item $M$ has positive rank over $R$, $R$ is unmixed, and $c < 1+ \dfrac{1}{\e_R(I,\,N)}$.
\item $M$ has positive rank over $R$, $R$ is unmixed, $\e_R(I, N) = r \cdot \e(R)$ for some $r \in \NN$, and $c < 1+ \dfrac{1}{r}$. 
\end{enumerate}
\end{thm}

\begin{proof} The exact sequence $0 \to \Omega^1_R(N) \to R^{\oplus \mu_R(N)} \to N \to 0$ induces an exact sequence
\[\Tor_1^R(M, N) \to M \otimes_R \Omega^1_R(N) \to M^{\oplus \mu_R(N)}  \xrightarrow{f} M\otimes_RN \to 0,
\]
which then yields following exact sequences
\[
\Tor_1^R(M, N) \to M \otimes_R \Omega^1_R(N) \to \ker(f) \to 0
\quad \text{and}\quad
0 \to \ker(f) \to M^{\oplus \mu_R(N)}  \xrightarrow{f} M\otimes_RN \to 0. 
\]
By Lemma~\ref{lem:c}, we have $\e_R(I, M^{\oplus \mu_R(N)}) = \e_R(I,M) \cdot \mu_R(N) \le c \cdot \e_R(I,M \otimes_R N)$.

(i) Suppose that $N$ is not free over $R$. Then $\rank(\Omega^1_R(N)) = \mu_R(N) - \rank(N) > 0$, which implies $\Supp_R(\Omega^1_R(N)) = \Spec(R)$. Localizing at the prime ideal $\fp \in \Assh(M)$ such that $\Tor_1^R(M,N)_{\fp}=0$, we get the exact sequence
$0 \to M_\fp \otimes_{R_{\p}} \Omega^1_R(N)_\fp \xrightarrow{\cong} \ker(f)_\fp \to 0$,
so $\ker(f)_\fp \cong M_\fp \otimes_{R_{\p}} \Omega^1_R(N)_\fp \neq 0$. Thus $\fp \in \Supp(\ker(f))$, hence $\dim_R(\ker(f)) = \dim_R(M)$. Consequently, 
\[
\e_R(I, \ker(f)) = \e_R(I,M^{\oplus \mu_R(N)}) - \e_R(I,M \otimes_R N) 
\le (c-1) \cdot \e_R(I,M \otimes_R N) < 1,
\]
which contradicts the fact that $\e_R(I, \ker(f)) \ge 1$.

(ii) In this case, we actually have $\Tor_1^R(M,N)_{\fp}=0$ for all $\fp \in \Assh_R(M)$. Suppose that $N$ is not free, so $\mu_R(N) > r$. As in the proof of (i) above, we see $\dim_R(\ker(f)) = \dim_R(M)$. By Lemma~\ref{lemrk}, $N$ has positive rank $r$ on $\Assh(M)$. Thus, 
\begin{align*}
\e_R(I, \ker(f)) &= \e_R(I,M^{\oplus \mu_R(N)}) - \e_R(I,M \otimes_R N) \\
&\le (c-1) \cdot \e_R(I,M \otimes_R N) 
< \frac{\e_R(I,\,M \otimes_R N)}{r} = \e_R(I, M).
\end{align*}
On the other hand, as $\mu_R(N) - r \ge 1$, we have
\begin{align*}
\e_R(I, \ker(f)) &= \e_R(I,M^{\oplus \mu_R(N)}) - \e_R(I,M \otimes_R N) \\
&= \e_R(I,M) \cdot \mu_R(N) - \e_R(I,M) \cdot r = \e_R(I, M) \cdot (\mu_R(N) - r) \ge \e_R(I, M),
\end{align*}
which is a contradiction.

(iii) In light of Proposition~\ref{High Bridge Lemma} (i), (iii), and (iv), it suffices to verify that $f: M^{\oplus \mu_R(N)} \onto M \otimes_R N$ is an isomorphism. Suppose, on the contrary that $\ker(f) \neq 0$. Since $M$ is unmixed, we must have $\dim(\ker(f)) = \dim(M)$. So, as in (i) above, we see $\e_R(I, \ker(f)) < 1$, which is a contradiction.

(iv) Since $M$ has positive rank over $R$, we see $\e_R(I, M) = \rank(M) \cdot \e_R(I,R)$. The assumption $\dim(M\otimes_R N) = \dim(M) = \dim(R)$ implies $\dim(N) = \dim(R)$. Thus $\e(I, M \otimes_R N) = \rank(M) \cdot \e_R(I,N)$. Now, considering the assumption $\e_R(I,M) \cdot \mu_R(N) \le c \cdot \e_R(I,M \otimes_R N)$, we find that $\e_R(I,R) \cdot \mu_R(N) \le c \cdot \e_R(I,N)$. As $R$ is unmixed, case (iii)(a) above applies to $M=R$, which implies that $N$ is free over $R$.

(v) Consider the syzygy exact sequence $0 \to \Omega^1_R(N) \to R^{\oplus \mu_R(N)} \to N \to 0$. Suppose that $N$ is not free over $R$, so that $\Omega^1_R(N) \neq 0$. As $R$ is unmixed, we see $\dim(\Omega^1_R(N)) = \dim(R)$. As in the proof of (iv) above, we see $\e_R(I,R) \cdot \mu_R(N) \le c \cdot \e_R(I,N)$, which then yields
\begin{align*}
\e_R(I, \Omega^1_R(N)) &= \e_R(I,R^{\oplus \mu_R(N)}) - \e_R(I,N) \\
& \le (c-1) \cdot \e_R(I,N)
= (c-1) \cdot r \cdot \e_R(I,R) < \e_R(I, R).
\end{align*}
On the other hand, we have
\begin{align*}
\e_R(I, \Omega^1_R(N)) &= \e_R(I,R^{\oplus \mu_R(N)}) - \e_R(I, N) \\
&= \e_R(I,R) \cdot \mu_R(N) - \e_R(I,R) \cdot r = \e_R(I, R) \cdot (\mu_R(N) - r),
\end{align*}
which implies that, as a positive integer, $\e_R(I, \Omega^1_R(N))$ is a multiple of $\e_R(I,R)$. This readily forces the inequality $\e_R(I, \Omega^1_R(N)) \ge \e_R(I, R)$, which is a contradiction.
\end{proof}

\begin{rmk}\label{rmk:alpha}
Theorem \ref{mainthm:alpha} can be stated in even more generality. Indeed, let $M$ and $N$ be as in Theorem~\ref{mainthm:alpha}. Denote 
\[\al(I, M, N) = \max\left\{\gcd\big(\e_R(I,\, M^{\oplus \mu_R(N)}),\, \e_R(I,\,M \otimes_R N)\big),\;\min\big\{\e_R(I, R/\fp) \mid \fp \in \Assh(M)\big\}\right\}.
\] 
From the proof of Theorem~\ref{mainthm:alpha} above, it is not hard to conclude that, in each of the cases listed in Theorem~\ref{mainthm:alpha} above, if $c < 1+ \frac{\al(I,\,M,\,N)}{\e_R(I,\,M \otimes_R N)}$, then $N$ is free. 
\end{rmk}

Next is our second main Theorem. It follows as a consequence of Theorem \ref{mainthm:alpha} and establishes Theorem \ref{mainthmintro} advertised in the introduction. 

\begin{thm} \label{mainthm} Let $M \neq 0$ and $N$ be $R$-modules. Assume that $M/IM$ and $N/IN$ are both free over $R/I$, $M\otimes_RN$ is Cohen-Macaulay, $\dim_R(M)=\dim_R(M\otimes_RN)$, and $\e_R(I,M)=\length_R(M/IM)$ (e.g., $M$ is Ulrich with respect to $I$). Then $N$ is free if at least one of the following holds:
\begin{enumerate}[\rm(i)]
\item $N$ has rank and  $\Tor_1^R(M,N)_{\fp}=0$ for \emph{some} $\fp \in \Assh_R(M)$.
\item $M$ is unmixed and at least one of the following holds:
\begin{enumerate}[\rm(a)]
\item $M$ is faithful.
\item $\Tor_1^R(M,N)=0$.
\item $\Tor_1^R(M,N)$ is torsion and $\Supp_R(M)=\Spec(R)$.
\end{enumerate}
\item $M$ has positive rank, and $R$ is unmixed.
\end{enumerate}
\end{thm}

\begin{proof} The cases (i), (ii), and (iii) follows from Theorem~\ref{mainthm:alpha} (i), (iii), and (iv), respectively.
\end{proof}

The following examples demonstrate several angles along which Theorem~\ref{mainthm} is sharp. The first example highlights that the assumption $\dim_R(M)=\dim_R(M \otimes_R N)$ cannot be removed:

\begin{eg} Let $R=k[\![x,y]\!]/(xy)$, $M=R/(x)$, and let $N=R/(y)$. Then $M$ is Ulrich, $M\otimes_RN\cong k$, $\Tor_1^R(M,N)=0$ and $\dim_R(M)=1\neq 0 = \dim_R(M\otimes_RN)$. 
\end{eg}

The next example shows that the hypothesis that $M$ is Ulrich is needed in Theorem \ref{mainthm}:

\begin{eg} Let $R=k[\![x,y]\!]/(x^2, xy)$, $M=R/(y)$ and $N=R/(x)$. Then $M\otimes_RN\cong k$ so $M\otimes_RN$ is Cohen-Macaulay. Moreover $\dim_R(M)=0= \dim_R(M\otimes_RN)$ and $\Tor_1^R(M,N)=0$. Note that $\mu_R(M)=1<2=\e_R(M)$ and $N$ is not free. 
\end{eg}

The next example shows that the torsion assumption on $\Tor_1^R(M,N)$ is needed in Theorem \ref{mainthm}:

\begin{eg} Let $R=k[\![x,y]\!]/(x^2)$, $M=N=R/(x)$. Then $M\otimes_R N\cong M$ are maximal Cohen-Macaulay $R$-modules with $\Supp_R(M)=\Spec(R)$ and $\mu_R(M)=1=\e_R(M)$ while $N$ is not free. On the other hand, we have $\Tor_i^R(M,N) \cong M$ for all $i\geq 0$, making $\Tor_1^R(M,N)$ torsion-free.
\end{eg}

Next is an example where Theorem \ref{mainthm:alpha} applies, but Theorem \ref{mainthm} does not. 

\begin{eg} \label{ex:sharp-4/3} Let $R=k[\![x,y,z]\!]/(xy,z^2)$, so that $M=\m$ is a $\frac {4}{3}$-Ulrich module; see Example~\ref{ex:4/3}. If $M \otimes_R N$ is Cohen-Macaulay with $\rank_R(N) = 1$ or $2$, then $N$ must be free by Theorem~\ref{mainthm:alpha}(ii). Thus, for any $\fm$-primary ideal $\fa$ of $R$ that is not principal, $\fm \otimes_R \fa$ must have torsion. 
\end{eg}


\section{On the test and rigidity properties of \texorpdfstring{$c$}{c}-Ulrich modules}

In this section, we prove some Corollaries of our main theorems on the test and rigidity properties of $c$-Ulrich module. Each of these results can be established in more generality, in view of Theorems \ref{mainthm:alpha} and \ref{mainthm} and Remark \ref{rmk:alpha}, but as these versions carry a great deal of technicality, we propose simpler versions. In particular, we focus on modules that are $c$-Ulrich with respect to the maximal ideal, rather than an arbitrary $\m$-primary ideal $I$.

It follows from work of \cite{JL07} (cf. \cite{JL07-E}) that any $c$-Ulrich module for $c<2$ is a test module for finite projective dimension in the sense of \cite{CD14}. Our next result shows that, at the expense of mild hypotheses on rank, the number of vanishings required can be substantially reduced.

\begin{cor}\label{pdtest:alpha} Let $R$ be a local ring, and let $M, N$ be $R$-modules such that $M$ is $c$-Ulrich for some real number $c<2$. Suppose that either $M$ has positive rank or that $N$ has rank and is free on $\Ass_R(M)$. Set $n=\dim_R(M)$, $v=\depth_R(M \otimes_R N)$,  and 
\[t=\begin{cases}\min\left\{s-1 \in \mathbb{Z}_{\geq 0} \mid \rank_R\big(\Omega^{n-v}_R N \big) \cdot (c-1)^s<1\right\}, & \text{ if $N$ has rank and is free on $\Ass(M)$}. \\[8pt] 
\min\left\{s-1 \in \mathbb{Z}_{\geq 0} \mid \e_R\big(\Omega^{n-v}_RN\big) \cdot (c-1)^s<1\right\}, & \text{ otherwise.} 
\end{cases}\] 
If $\Tor^R_i(M,N)=0$ for all $i=1,\dots,n-v+t$,  then $\pd_R(N)<\infty$. 
\end{cor}

\begin{proof}
Assume $\Tor^R_i(M,N) = 0$ for $i=1,\dots,n-v+t$. Then there are exact sequences
\begin{equation}
\tag{\ref*{pdtest:alpha}.1}\label{tor-syzygy}
0 \to M \otimes_R \Omega^i_R(N) \to M^{\oplus \beta^R_{i-1}(N)} \to M \otimes_R \Omega^{i-1}_R(N) \to 0
\end{equation}
for all $i=1, \ldots,  n-v+t$. By applying the depth lemma to \eqref{tor-syzygy}, we conclude that $M \otimes_R \Omega^i_R(N)$ is Cohen-Macaulay of dimension $n$ for all $i=n-v, \ldots, n-v+t$. Inductively, we see 
\[
\tag{\ref*{pdtest:alpha}.2}\label{e-syzygy}
\e(M \otimes_R \Omega^{n-v+i}_R(N)) \le (c-1)^i \cdot \e(M \otimes_R \Omega^{n-v}_R(N)) \quad \text{for all} \quad i = 0, \,1,\, \dotsc,\, t.
\]

When $N$ has rank and is free on $\Ass(M)$, the same is true for $\Omega_R^{n-v+t}(N)$. If $\Omega_R^{n-v+t}(N)=0$, there is nothing to prove. So we assume $\Omega_R^{n-v+t}(N) \ne 0$ so that $M \otimes_R \Omega_R^{n-v+t} \ne 0$. Since $M$ is Cohen-Macaulay, we see $\dim(M \otimes_R \Omega_R^{n-v+t})=\dim(M)$ from \ref{tor-syzygy}. Also, \ref{e-syzygy} indicates \[r:=\rank(\Omega_R^{n-v+t}(N)) \le (c-1)^t \cdot \rank(\Omega_R^{n-v}(N)).\]
Hence $(c-1)^{t+1} \cdot \rank(\Omega^{n-v}_R(N))<1$ which forces $c<1+\frac{1}{r}$. By Theorem~\ref{mainthm:alpha}(ii), this forces $\Omega^{n-v+t}_R(N)$ to be free, so that $\pd_R(N)<\infty$. 

In the case where $M$ has positive rank, it follows that $R$ is unmixed since $M$ is Cohen-Macaulay, and \eqref{e-syzygy} implies \[\e_R(\Omega^{n-v+t}_R(N)) \le (c-1)^t \cdot \e(\Omega^{n-v}_R(N)).\] We thus have $(c-1)^{t+1} \cdot \e_R(\Omega^{n-v}_R(N)) < 1$. Therefore $c-1 < \frac{1}{\e_R(\Omega^{n-v+t}_R(N))}$ which yields the freeness of $\Omega^{n-v+t}_R(N)$ in light of Theorem~\ref{mainthm:alpha}(iv). 
\end{proof}

Next, we show that $c$-Ulrich modules for $c<2$ can be used to test injective dimension under mild hypotheses. We will use the following technical lemma, which is a direct improvement over \cite[3.4 (1)-(2)]{LM19}. When $R$ is Cohen-Macaulay with canonical module $\w_R$, we set $(-)^{\dagger}:=\Hom_R(-,\w_R)$. 

\begin{lem} \label{J} 
Let $R$ be $d$-dimensional Cohen-Macaulay ring with canonical module $\w_R$, let $M$ be an $R$-module, and let $N$ be a maximal Cohen-Macaulay $R$-module. For some $n \ge d$, suppose $\depth(\Ext_R^{n-i}(M, N)) \geq \min\{i+1, d\}$ for all $i = 0,\dotsc, n-1$ \footnote{By convention, we regard the zero module to have infinite depth, and consider it to be maximal Cohen-Macaulay.}. Then $\Tor_i(M, N^{\dagger}) \cong \Ext_R^i(M,N)^{\dagger}$ for all $i=0,\dots,n-d$ and they are maximal Cohen-Macaulay. In particular, $M \otimes_R N^{\dagger}$ is maximal Cohen-Macaulay.
\end{lem}

\begin{proof} Let $\cdots \rightarrow F_{i+1} \xrightarrow{d_i} F_i \rightarrow \cdots \rightarrow F_1 \xrightarrow{d_0} F_0 \rightarrow 0$ be a minimal free resolution of $M$, and set $\partial_i:=\Hom_R(d_i,N)$ for each $i$. Applying $\Hom_R(-,N)$, we have a complex
\begin{equation}
\tag{\ref*{J}.1}\label{dualcomplex}0 \to \Hom_R(M,N) \to \Hom_R(F_0,N) \to \cdots \to \Hom_R(F_{n-d+1},N) \to \im \partial_{n-d+1} \to 0\end{equation}
whose homologies, $\Ext^i_R(M,N)$ for $i=1,\dots,n-d$, are maximal Cohen-Macaulay by assumption. We claim $\im \partial_{i}$ is maximal Cohen-Macaulay for $i=0,\dots,n-d+1$ and that $\ker \partial_{i}$ is maximal Cohen-Macaulay for $i=0,\dots,n-d$. Indeed, for each $0 \le i \le n$, we have short exact sequences 
\[\tag{\ref*{J}.2}\label{ses1}
0 \to \im \partial_{i} \to \ker \partial_{i+1} \to \Ext^{i}_R(M,N) \to 0
\quad\text{and}\quad
0 \to \ker \partial_{i} \to \Hom_R(F_{i},N) \to \im \partial_{i} \to 0.
\]
As $\im \partial_{n} \hookrightarrow \Hom_R(F_{n+1},N)$, we have $\depth (\im \partial_{n}) \ge 1$. Applying the depth lemma repeatedly to the exact sequences in \eqref{ses1}, 
we see that $\im \partial_{i}$ is maximal Cohen-Macaulay for $i=0,\dots,n-d+1$ and that $\ker \partial_{j}$ is maximal Cohen-Macaulay for $j=0,\dots,n-d$. As $\Hom_R(M,N) \cong \ker \partial_0$, we have in particular that $\Hom_R(M,N)$ is maximal Cohen-Macaulay as well.

We have then obtained that every term, cycle, boundary, and homology of \eqref{dualcomplex} is maximal Cohen-Macaulay. Thus applying $(-)^{\dagger}$ to \eqref{dualcomplex}, we obtain another complex whose homologies are $\Ext^i_R(M,N)^{\dagger}$ for $i=1,\dots,n-d$. This complex may be identified, via Hom-tensor adjointness, with the complex
\[0 \to (\im \partial_{n-d+1})^{\dagger} \to (F_{n-d+1} \otimes_R N^{\dagger})^{\dagger \dagger} \to \cdots \to (F_0 \otimes_R N^{\dagger})^{\dagger \dagger} \to \Hom_R(M,N)^{\dagger} \to 0\]
Since $N$ is maximal Cohen-Macaulay, the natural maps induce an isomorphism of complexes
\[\begin{tikzcd}[column sep=small]
	0 & {F_{n-d+1} \otimes_R N^{\dagger}} & \cdots & {F_1 \otimes_R N^{\dagger}} & {F_0 \otimes_R N^{\dagger}} & 0 \\
	0 & {(F_{n-d+1} \otimes_R N^{\dagger})^{\dagger \dagger}} & \cdots & {(F_1 \otimes_R N^{\dagger})^{\dagger \dagger}} & {(F_0 \otimes_R N^{\dagger})^{\dagger \dagger}} & 0
	\arrow[from=1-1, to=1-2]
	\arrow[from=1-2, to=1-3]
	\arrow[from=1-3, to=1-4]
	\arrow[from=1-4, to=1-5]
	\arrow[from=1-5, to=1-6]
	\arrow[from=2-1, to=2-2]
	\arrow[from=2-2, to=2-3]
	\arrow[from=2-3, to=2-4]
	\arrow[from=2-4, to=2-5]
	\arrow[from=2-5, to=2-6]
	\arrow[from=1-2, to=2-2]
	\arrow[from=1-4, to=2-4]
	\arrow[from=1-5, to=2-5]
\end{tikzcd}\]

By above, the bottom row has homologies $\Ext^i_R(M,N)^{\dagger}$ in degrees $i=0,\dots,n-d$. Comparing with the homologies of the top row, it follows that $\Tor^R_i(M,N^{\dagger}) \cong \Ext^i_R(M,N)^{\dagger}$ for $i=0,\dots,n-d$, and the claim follows. 
\end{proof}

It follows immediately from combining Lemma \ref{J} with \cite[2.1]{JL07-E} that $c$-Ulrich modules for $c<2$ are test modules for finite injective dimension. But as with Corollary \ref{pdtest:alpha}, our methods allow us to greatly reduce the number of vanishings required, at the expense of mild hypotheses on rank.

We first consider the test behavior of Ulrich modules for finite injective dimension.

\begin{cor} \label{corExt1} Let $R$ be a $d$-dimensional Cohen-Macaulay local ring and let $M$ and $N$ be $R$-modules. Assume the following conditions hold:
\begin{enumerate}[\rm(i)]
\item $M$ is Ulrich.
\item $M$ or $N$ is maximal Cohen-Macaulay. 
\item $\Ext^i_R(M,N)=0$ for all $i=1, \ldots, d$.
\item Either $(a)$ $\depth_R\big(\Ext^{d+1}_R(M,N)\big)>0$, or $(b)$ $\Supp_R(M)=\Spec(R)$ and $d>0$.
\end{enumerate}
Then $\id_R(N)<\infty$.
\end{cor}

\begin{proof} We may assume $R$ is complete so that $R$ has a canonical module $\omega_R$. Consider the case where $M$ is maximal Cohen-Macaulay. Then we take a maximal Cohen-Macaulay approximation $0 \to Y \to L \to N \to 0$, so that $\id_R(Y)<\infty$ and $L$ is maximal Cohen-Macaulay (see \cite[11.8]{LW12}). Since $M$ is maximal Cohen-Macaulay and $\id_R(Y)<\infty$, we have $\Ext^i_R(M,Y)=0$ for all $i>0$, and $\id_R(L)<\infty$ if and only if $\id_R(N)<\infty$ \cite[11.3]{LW12}. We may thus suppose that $N$ is maximal Cohen-Macaulay.

Now assume part (a) of condition (iv) holds. Then $\Tor_1^R(M,N^{\dagger})=0$ by Lemma \ref{J}. Hence Theorem \ref{mainthm}(ii)(b) shows that $N^{\dagger}$ is free. 

Next assume part (b) of condition (iv) holds. Let $\fp \in \Ass(R)$. As $\Ext^1_{R_{\fp}}(M_{\fp}, N_{\fp})=0$, we use Lemma \ref{J} with $M_{\fp}$ and $N_{\fp}$ over the Artinian ring $R_{\fp}$, and conclude that $\Tor_1^R(M,N^{\dagger})_{\fp}=0$. This shows that $\Tor_1^R(M,N^{\dagger})$ is torsion. Thus Theorem \ref{mainthm}(ii)(c) shows that  $N^{\dagger}$ is free.
\end{proof}

\begin{defn} Let $R$ be a Cohen-Macaulay local ring and let $M$ be an $R$-module. We say $M$ has dual rank $r$, denoted by $\drank_R(M)=r$, if $\widehat{M}_{\fp}\cong \big(\omega_{\widehat{R}}\big)^{\oplus r}_{\fp}$ for all $\fp \in \Ass(\widehat{R})$, where $\hat{R}$ denotes the $\m$-adic completion of $R$ and $\w_{\hat{R}}$ is its canonical module. 

Note that if $M$ has rank, then $M$ has dual rank if and only if $\hat{R}$ is generically Gorenstein, in which case the rank and dual rank of $M$ agree. 
\end{defn}

We now consider the behavior for $c$-Ulrich modules.

\begin{cor} \label{corExt2} Let $R$ be a $d$-dimensional Cohen-Macaulay local ring and let $M$ and $N$ be $R$-modules. Assume that the following conditions hold:
\begin{enumerate}[\rm(i)]
\item $M$ is $c$-Ulrich for some $c$ where 
\[c<\begin{cases} 
1+ \dfrac{1}{\drank(N)}, & \text{ if $N$ has positive dual rank.} \\[13pt]
1+ \dfrac{1}{\rank(N)}, & \text{ if $M$ and $N$ both have positive rank.} \\[13pt]
1+ \dfrac{1}{\e_R(N)}, & \text{ if $M$ has positive rank, but $N$ does not have positive rank nor dual rank.}

\end{cases}\]

\item $N$ is maximal Cohen-Macaulay. 
\item $\Ext^i_R(M,N)=0$ for all $i=1, \ldots, d$.
\end{enumerate}
Then $\id_R(N)<\infty$.
\end{cor}

\begin{proof} We may assume $R$ is complete so that $R$ has a canonical module $\omega_R$. By Lemma~\ref{J}, $M \otimes_R N^{\dagger}$ is maximal Cohen-Macaulay, which forces $\dim M=\dim R$, so that $M$ is maximal Cohen-Macaulay. In the case where $N$ has positive dual rank, we see $\rank(N^{\dagger}) = \drank(N) >0$; otherwise $M$ has positive rank and $\e_R(N^{\dagger})=\e_R(N)$, where $\e_R(N) = \rank_R(N)\e_R(R)$ if $N$ has rank. In each case, Theorem \ref{mainthm:alpha} implies $N^{\dagger}$ is free, so that $\id_R(N)<\infty$. 
\end{proof}

\begin{rmk}
As in Corollary \ref{corExt1}, the hypotheses that $N$ be maximal Cohen-Macaulay in Corollary \ref{corExt2} can be exchanged for the hypotheses that $M$ be maximal Cohen-Macaulay, but at the expense of decreasing the value of $c$. Indeed, as in the proof of Corollary \ref{corExt1}, one may take a maximal Cohen-Macaulay approximation $0 \to Y \to L \to N \to 0$ of $N$. The hypothesis on $\Ext$-vanishing will force $\Ext^i_R(M,L)=0$ for a $i=1,\dots,d$, and we have $\id_R(L)<\infty$ if and only if $\id_R(N)<\infty$. So the conclusion that $\id_R(N)<\infty$ will still hold in this situation if $c$ is assumed to be less than $1+\dfrac{1}{\drank(L)}$ in the case when $N$ has positive dual rank, $1+\dfrac{1}{\rank(N)+\drank(Y)}$ in the case when both $M$ and $N$ have positive rank, and $1+\dfrac{1}{\e_R(L)}$ when $M$ has positive rank and $N$ does not have positive rank nor dual rank.
\end{rmk}

Corollary \ref{corExt2} generalizes a well-known result of Ulrich \cite[3.1]{Ul84} that served as a jumping off point historically for the theory of Ulrich modules, and that has been revisited several times in the literature \cite[3.4]{HH05}, \cite[6.8]{LM19}, \cite[2.2, 2.4]{JL07} (cf. \cite[2.2, 2.4]{JL07-E}). Our approach unifies several of these results, and allows for more flexibility in hypotheses. 

\begin{cor} \label{BU} Let $R$ be a $d$-dimensional Cohen-Macaulay local ring and let $M$ be an $R$-module that is $c$-Ulrich for some $c<2$. Assume that either $M$ has positive rank or $\wh R$ is generically Gorenstein.  If $\Ext^i_R(M,R)=0$ for all $i=1, \ldots, d$, then $R$ is Gorenstein. 
\end{cor}

The case where $M$ has positive rank in Corollary \ref{BU} recovers \cite[3.1]{Ul84}, while the generically Gorenstein case sharpens \cite[2.4]{JL07-E} (when $M$ is maximal Cohen-Macaulay). 

Corollary \ref{corExt2} is easily seen to be sharp, as one can take $R$ to be any Cohen-Macaulay local ring of multiplicity $2$ and take $M=R$. Then $\Ext^i_R(M,N)=0$ for all $i>0$ and any $R$-module $N$, but there are modules of infinite injective dimension, since $R$ has multiplicity $2$ and thus cannot be regular. It is more difficult to see that Corollary \ref{BU} is sharp, due to the fact that if $R$ is Cohen-Macaulay with $\e(R)=2$, then $R$ must be Gorenstein.  The existence of such a nonfree $M$ withnessing this sharpness is significant in light of the peviously mentioned results of Ulrich and others in the literature. We provide the following general construction for producing $c$-Ulrich modules in dimension $1$, for controlled values of $c$, which have $\Ext^1_R(M,R)=0$. Using this construction, we provide Example \ref{BUex} demonstrating the sharpness of Corollaries \ref{corExt2} and \ref{BU} in a strong way; we present a nonfree $2$-Ulrich module $M$ over a one-dimensional complete domain of minimal multiplicity with $\Ext^1_R(M,R)=0$, but with $R$ not Gorenstein. 

We recall that if $N$ is an $R$-module with minimal presentation $R^{\oplus b} \xrightarrow{A} R^{\oplus a} \rightarrow N \rightarrow 0$, the \emph{Auslander transpose} $\Tr_R(N)$ of $N$ is the $R$-module presented by the transpose matrix $A^T$. 

\begin{prop} \label{BUprop} Let $R$ be a one-dimensional Cohen-Macaulay local ring, $N$ be a maximal Cohen-Macaulay $R$-module with rank such that $\mu_R(N)=\rank_R(N)+1$, and let $M= \Tr_R(N) / \Gamma_{\fm}\Tr_R(N)$. Then $M$ is a nonzero maximal Cohen-Macaulay $R$-module such that $\e_R(M)=\e(R) \cdot (\beta_1^R(N)-1)$, $\mu_R(M)=\beta_1^R(N)\geq 2$, and $\Ext^1_R(M,R)=0$. Moreover, if $\e(R)\geq 4$, then $\e_R(M) \geq 2 \cdot \mu_R(M)$.
\end{prop}

\begin{proof} Set $a=\mu_R(N)$ and $b=\beta_1^R(N)$. Then there is an exact sequence $R^{\oplus b} \to R^{\oplus a} \to N \to 0$. This exact sequence yields the exact sequence $0 \to N^{\ast} \to R^{\oplus a} \to R^{\oplus b} \to \Tr_R N \to 0$. Hence we have $\rank_R(\Tr_R(N)) = \rank_R(N^{\ast})+(b-a)$. 
Set $L=\Gamma_{\fm}\Tr_R(N)$ and consider the short exact sequence $0 \to L \to \Tr_R N \to M \to 0$. Applying $\Hom_R(-, R)$ to the exact sequence, we obtain the exact sequence $L^* \to \Ext^1_R(M,R) \to \Ext^1_R(\Tr_R N, R)$, where $L^* := \Hom_R(L,R) =0$ as $\length_R(L)<\infty$. Also note that $\Ext^1_R(\Tr_R N, R)=0$ (as $\Ext^1_R(\Tr_R N, R)$ is the kernel of the canonical map $N \to N^{**}$). This forces $\Ext^1_R(M,R)=0$. Note that, since $\depth(R)=1$, $\Ext^1_R(L,R) \ne 0$; see \cite[1.2.10(e)]{BH93}. Hence $L \ncong \Tr_R(N)$ which implies $M \ne 0$, due to the exact sequence $0 \to L \to \Tr_R(N) \to M \to 0$. Moreover, since $\rank_R(\Tr_R(N))=\rank_R(L)+\rank_R(M)=\rank_R(M)$ and $\rank_R(N^{\ast})=\rank_R(N)$, we conclude that $$\rank_R(M)=\rank_R(N)+(b-a)=(a-1)+(b-a)=b-1.$$
Hence $\e_R(M)=\e(R) \cdot \rank_R(M)=\e(R) \cdot (b-1)$. Note that $b\geq 2$ since $0\neq M$ is torsion-free.

The surjections $R^{\oplus b} \to  \Tr_R(N)$ and $\Tr_R(N) \to M$ imply that $$b\geq \mu_R(\Tr_R(N)) \geq \mu_R(M) \geq \rank_R(M)=b-1.$$ Hence $\mu_R(M)$ is either $b-1$ or $b$. If $\mu_R(M)=b-1$, then $\mu_R(M)=\rank_R(M)$, which forces $M$ to be free; If $M$ is free then $\Tr_R(N) \cong L \oplus M$ and so the vanishing of $\Ext^1_R(\Tr_R N, R)$ forces $\Ext^1_R(L,R)=0$, contradicting $\Ext^1_R(L,R)\neq 0$. Therefore $\mu_R(M)=b$.
As $b\geq 2$, one can observe that, if $\e(R)\geq 4$, then $\e_R(M) = \e(R) \cdot (b-1) \geq 2 \cdot \mu_R(M)$.
\end{proof}

\begin{eg} \label{BUex} Let $R=k[\![t^3, t^4, t^5]\!]$. Then $R$ is a one-dimensional local domain with canonical module $\omega_R=(t^3, t^4)$ and $\e(R)=3$. It follows that $\mu_R(\omega_R)=2$ and $\beta_1^R(\omega_R)=3$. Set $N=\omega_R$ and $M=\Tr_R(N) / \Gamma_{\fm}\Tr_R(N)$. Then, by Proposition~\ref{BUprop}, $M$ is a maximal Cohen-Macaulay $R$-module such that $\Ext^1_ R(M,R)=0$ and $\e_R(M)=6$ and $\mu_R(M)=3$. So $M$ is $2$-Ulrich, but $R$ is not Gorenstein.
\end{eg}

Example \ref{BUex} also demonstrates that the vanishing of $\Ext^{d+1}_R(M,R)$ in \cite[6.8]{LM19} cannot be removed.

Proposition \ref{BUprop} allows us to obtain examples similar to \ref{BUex} whenever $R$ is a local domain such that $\e(R)\geq 4$ and $I$ is an ideal of $R$ such that $\mu_R(I)=2$.

\begin{eg} Let $R=k[\![t^5, t^6, t^8, t^9]\!]$ and let $M= \Tr_R(N) / \Gamma_{\fm}\Tr_R(N)$, where $N=(t^5, t^6)R$. Then $R$ is not Gorenstein and $M$ is a nonzero maximal Cohen-Macaulay $R$-module with $\Ext^1_R(M,R)=0$. Also one can check that $M$ is $(15/4)$-Ulrich since $\e_R(M)=\e(R) \cdot (\beta_1^R(N)-1)=5 \cdot(4-1)=15$ and $\mu_R(M)=\beta_1^R(N)=4$; see Proposition \ref{BUprop}.
\end{eg}

Next, we show if $c$ is smaller than a specific threshold, then the rigidity of a $c$-Ulrich module $M$ sharpens to an extreme, and $M$ can be used as a test module for the same homological dimensions detected by $k$. More precisely, we will show the following:

\begin{thm} \label{cUlrichtest} Let $R$ be a local ring and let $M$ be a $c$-Ulrich module for some $c<1+\frac{1}{v}$, where $v=\codim(R/\Ann_R(M))$. 
\begin{enumerate}[\rm(i)]
\item If $|k| = \infty$, then there is an $M$-regular sequence $\{\underline{x}\}\subseteq \fm$ such that $k$ is a direct summand of $M/\underline{x}M$. 
\item If $N$ is an $R$-module and $\Ext^{i}_R(N,M)=0$ for all $i=n, \ldots, n+\dim_R(M)$ for some $n\geq 0$, then $\pd_R(N)<\infty$.
\end{enumerate}
\end{thm}

In order to prove Theorem \ref{cUlrichtest}, we prepare a lemma which will also be used later to establish Theorem~\ref{Ulrichgrowtoseh}.

\begin{lemma}\label{lengthinq} Let $R$ be a local ring and let $M,N$ be $R$-modules with $\length_R(N)<\infty$.
\begin{enumerate}[\rm(i)]
\item We have $\length_R(\Ext^i_R(M,N)) \le \beta^R_i(M) \cdot \length_R(N)$ for all $i\geq 0$.
\item If $\mu_R(N) > \len_R(\m N) \cdot \mu_R(\m)$, then there is a socle element of $N$ which is a minimal generator of $N$, and hence $k$ is a direct summand of $N$.  
\end{enumerate}
\end{lemma}

\begin{proof} (i) We prove the claimed inequality by induction on $\length_R(N)$. Fix $i\geq 0$. If  $\length_R(N)=1$, then $N \cong k$, and so we have $\length_R(\Ext^i_R(M,N))=\length_R(\Ext^i_R(M,k))=\beta^R_i(M)$.  Suppose $\length_R(N)>1$, and that the result holds for modules that have length less than the length of $N$. There is an exact sequence of $R$-modules of the form $0 \to k \to N \to C \to 0$.  By applying $\Hom_R(M,-)$ to this short exact sequence, we obtain: 
\begin{align*}
\length_R(\Ext^i_R(M,N)) &\le \displaystyle \length_R(\Ext^i_R(M,k))+\length_R(\Ext^i_R(M,C)) \\
&\le \displaystyle \beta^R_i(M)+\length_R(C) \cdot \beta^R_i(M) 
= \displaystyle \beta^R_i(M) \cdot (1+\length_R(C)) 
= \displaystyle \beta_i^R(M) \cdot \length_R(N). 
\end{align*} 

(ii) Suppose no minimal generator of $N$ is in $\Soc(N)$, hence $\Soc(\m N)=\Soc(N)$. Applying $\Hom_R(k,-)$ to the exact sequence $0 \to \m N \to N \to N/\m N \to 0$, we get an exact sequence
\[0 \to \Hom_R(k,\m N) \to \Hom_R(k,N) \to \Hom_R(k,N/\m N) \to \Ext^1_R(k,\m N).\]
As $\dim_k(\Hom_R(k,\m N))=\dim_k(\Soc(\m N))=\dim_k(\Soc(N))=\dim_k(\Hom_R(k,N))$, the first map is an isomorphism. Thus $\Hom_R(k,N/\m N) \hookrightarrow \Ext^1_R(k,\m N)$. Note that $\dim_k(\Hom_R(k,N/\m N))=\mu_R(N)$. Hence it follows from (i) that 
\[\mu_R(N) \le \length_R(\Ext^1_R(k,\m N)) \le \beta^R_1(k)\length_R(\m N)=\mu_R(\m) \cdot \length_R(\m N).
\qedhere
\] 
\end{proof}

The following explains the connection between Lemma \ref{lengthinq} and Theorem \ref{cUlrichtest}.

\begin{remark} Let $R$ be a Artinian local ring that is not a field. Set $v=\codim(R)=\mu_R(\fm)$. If $M$ is an $R$-module, then $\mu_R(M) > \len_R(\m M) \cdot v$ if and only if $\e_R(M)<\big(1+\frac{1}{v}\big) \cdot \mu_R(M)$ if and only if $M$ is $c$-Ulrich for some $c$ such that $c<1+\frac{1}{v} \le 2$.  
\end{remark}

\begin{proof}[Proof of Theorem \ref{cUlrichtest}]
(i) With $|k|=\infty$, we choose a system of parameters $\underline{x}$ of $M$ which generates a minimal reduction of the maximal ideal of $R/\Ann_R(M)$. Then $\underline{x}$ is $M$-regular and Lemma~\ref{lengthinq}(2) forces $k$ to be a direct summand of $M/\underline{x}M$.

(ii) If needed, we may extend the residue field and hence we may assume $|k|=\infty$. Then, by part (i), there is a regular sequence $\underline{x}$ on $M$ such that $k$ is a direct summand of $M/\underline{x}M$. The hypothesis on the vanishing of Ext shows $\Ext^{n}_R(N,M/\underline{x} M)=0$, hence $\Ext^{n}_R(N,k)=0$. So $\pd_R(N)<\infty$. 
\end{proof}

Theorem \ref{mainthm} gives a quick proof of  the following fact: if $R$ is a local ring and $N$ is an Ulrich $R$-module such that $\pd_R(N)<\infty$, then $R$ is regular. To see that, set $M=\Omega^n_R(k)$ for some $n\gg 0$. Then, as $\pd_R(N)<\infty$, $\Tor_i^R(M,N)=0$ for all $i\geq 1$. We see, by using the depth lemma, that $\dim_R(N)=\depth_R(N)=\depth_R(M\otimes_RN)\leq \dim_R(M\otimes_RN)\leq \dim_R(N)$.
So Theorem \ref{mainthm} implies that $M$ is free, and hence $R$ is regular.  Theorem \ref{cUlrichtest}(i) implies more, that any such module behaves in the same fashion with any suitable homological dimension. For instance we record the following:
\begin{cor} \label{HdimUlrich} Let $R$ be a local ring and let $M$ be a $c$-Ulrich module for some $c<1+\frac{1}{v}$ where $v=\codim(R/\Ann_R(M))$. If $M$ has finite projective or injective dimension (respectively, finite Gorenstein dimension), then $R$ is regular (respectively, $R$ is Gorenstein).
\end{cor}

\begin{proof} We may assume, without lost of generality, that $|k|=\infty$. Hence Theorem \ref{cUlrichtest}(i) shows that $k$ is a direct summand of $M/\underline{x}M$ for a general $M$-regular sequence $\x \subseteq \fm$. Hence the claims follows from \cite[2.2.7]{BH93}, \cite[3.1.26]{BH93}, and \cite[17]{Ma00}.
\end{proof}
\begin{cor} \label{RigidUlrich} Let $R$ be a one-dimensional local domain and let $M$ and $N$ be $R$-modules such that $M$ is Ulrich and $\Ext^1_R(M,N)=0$. Then $\id_R(N)<\infty$. Therefore, if $M=N$, then $M$ is free and $R$ is regular.
\end{cor}

\begin{proof} If $\depth_R(M)=0$, then $\length_R(M)<\infty$ so that $\mu_R(M)=\e_R(M)=\length_R(M)$ and hence $M \cong k^{\oplus n}$ for some $n\geq 1$. In that case, the vanishing of $\Ext^1_R(M,N)$ shows that $\id_R(N)<\infty$. Suppose $\depth_R(M)=1$. Then $M$ is maximal Cohen-Macaulay and so $M$ has positive rank. Therefore $\Supp_R(M)=\Spec(R)$ and hence Corollary \ref{corExt1} implies that $\id_R(N)<\infty$. 

If $\Ext^1_R(M,M)=0$, then $M$ is maximal Cohen-Macaulay since $\Ext^1_R(k,k)\neq 0$. Thus $\id_R(M)<\infty$ and Corollary \ref{HdimUlrich} shows that $R$ is regular, and hence $M$ is free.
\end{proof}


In \cite[2.1]{JL07} (cf. \cite[2.1]{JL07-E}), Jorgensen-Leuschke use a particular numerical deviation from a module being Ulrich to provide an upper bound on the growth of successive Betti numbers of an $R$-module $N$ under certain conditions on vanishing of $\Tor^R_i(M,N)$ or $\Ext^i_R(M,N^{\dagger})$. We provide a variation on their result by using this numerical deviation to provide a lower bound on the growth rate of the successive Betti numbers of $M$, with analogous conditions on the vanishing of $\Ext^i_R(M,N)$. In a sense, while their results are deduced from homological information, ours are cohomological in nature.

\begin{theorem}\label{Ulrichgrowtoseh} Let $R$ be a local ring and let $M$ and $N \neq 0$ be $R$-modules. Assume that $N$ is Cohen-Macaulay with $n = \dim(N)$.  If $\Ext^i_R(M,N)=0$ for $t \le i \le t+n$ for some $t\ge 1$, then 
$$\beta_{t+1}^R(M) \ge \dfrac{\mu_R(N)}{\e_R(N)-\mu_R(N)} \cdot \beta^R_t(M).$$
\end{theorem}

\begin{proof} Extending the residue field if needed, we may suppose that the residue field $k$ is infinite. 
There exists an $N$-regular sequence $\x = x_1,\dots,x_n \in \fm$ such that the ideal $(\x) =(x_1,\dots,x_n)$ is a minimal reduction of $\m$ modulo $\Ann(N)$ \cite[Section 4.6]{BH93}.  Denote $\bar{N} = N/\x N$.  As $\x$ is a regular sequence on $N$, it follows that $\Ext^t_{R}(M,\bar{N})=0$ \cite[Proof of 1.2.4]{BH93}. Applying $\Hom_{R}(M,-)$ to the short exact sequence 
$0 \to \m\bar{N} \to \bar{N} \to k^{\oplus\mu_R(N)} \to 0$, 
we get an injection
$\Ext^t_{R}(M,k^{\oplus\mu_R(N)}) \hookrightarrow \Ext^{t+1}_{R}(M,\m\bar{N})$.  This, combined with Lemma~\ref{lengthinq}(i) gives us that 
\[\length_R(\Ext^t_{R}(M,k^{\oplus\mu_R(N)})) \le \length_R(\Ext^{t+1}_{R}(M,\m\bar{N})) \le \beta_{t+1}^{R}(M) \cdot \length_R(\m\bar{N})=\beta_{t+1}^R(M)\cdot (\e_R(N)-\mu_R(N)).\]
As $\length_R(\Ext^t_{R}(M,k^{\oplus\mu_R(N)}))=\beta^R_t(M)\cdot \mu_R(N)$, the result follows.
\end{proof}

Applying Theorem \ref{Ulrichgrowtoseh} when $N$ is the canonical module, we immediately recover a general result of Avramov; see \cite[4.2.7]{Av10}. 

\begin{cor}\label{canonicalUlrich}
Let $R$ be a Cohen-Macaulay local ring and let $M$ be a nonfree maximal Cohen-Macaulay $R$-module. Then $\beta_{i+1}^R(M) \ge \dfrac{r(R)}{\e(R)-r(R)}\cdot \beta^R_i(M)$ for all $i \ge 1$.  In particular, 
\begin{enumerate}[\rm(i)]
\item If $\e(R) \le 2r(R)$, then the Betti sequence of any module is nondecreasing after at most $d$ steps. 
\item If $\e(R)<2r(R)$, the the Betti sequence of any module of infinite projective dimension has exponential growth after at most $d$ steps.
\end{enumerate}
\end{cor}

\begin{proof} We may suppose $R$ is complete so that it admits a canonical module $\w$.  As $\e_R(\w)=\e(R)$ and $\mu_R(\w)=r(R)$, the result follows from Theorem \ref{Ulrichgrowtoseh}.
\end{proof}

\begin{example} If $R=k[\![x,y]\!]/(x^2,xy,y^2)$, then $\mu_R(\w)=r(R)=2$ and $\e(\w)=\length(R)=3$ so $\w$ is $(3/2)$-Ulrich. Applying Corollary \ref{canonicalUlrich} is one (of many) ways to see that every nonfree module over $R$ has a Betti sequence that grows exponentially.
\end{example}

\section{Additional Applications}

In this section we record several additional consequences of our main theorems. We also provide additional examples to demonstrate their sharpness.

\begin{cor} \label{corR} Let $R$ be a local ring and let $M$ and $N$ be nonzero $R$-modules such that either $M$ is faithful or $N$ has rank. If $M$ is Ulrich and $M\otimes_RN$ is maximal Cohen-Macaulay, then $N$ is free.
\end{cor} 

\begin{proof} Note that, since $M$ is Cohen-Macaulay and $M\otimes_RN$ is maximal Cohen-Macaulay, $M$ is also maximal Cohen-Macaulay. So $\dim_R(M\otimes_RN)=\dim(R)=\dim_R(M)$ and this implies that $\Assh(M)\subseteq \Min(R) \subseteq \Ass(R)$.

In the case where $M$ is faithful, Theorem \ref{mainthm}(ii)(a) shows that $N$ is free.

Next assume $N$ has rank. Then $\Tor_1^R(M,N)_{\fp}=0$  for each $\fp \in \Assh(M)$. Thus Theorem~\ref{mainthm}(i) shows that $N$ is free. 
\end{proof}

\begin{cor} \label{corMM} Let $R$ be a $d$-dimensional Cohen-Macaulay local ring that has minimal multiplicity and let $M$ and $N$ be $R$-modules, either of which has rank. If $\Omega^n_R (M)\otimes_R N$ is nonzero and maximal Cohen-Macaulay for some $n\geq d+1$, then $N$ is free. 
\end{cor}

\begin{proof} Note that $\Omega^{n-1}_R(M)$ is maximal Cohen-Macaulay so that $\Omega^n_R(M) =\Omega^1_R (\Omega^{n-1}_R (M))$ is Ulrich \cite[1.6]{KT19}. Hence the result follows from Corollary \ref{corR}.
\end{proof}

As a consequence, we show that the Huneke-Wiegand conjecture (see \cite[page~473]{HW94}) holds true for Ulrich modules, or over rings with minimal multiplicity. The minimal multiplicity case recovers a result of Huneke-Iyengar-Wiegand, obtained through different methods \cite[3.5]{HI19}.

\begin{cor}\label{hunwie} Let $R$ be a one-dimensional Cohen-Macaulay local ring and let $M$ be a torsion-free $R$-module such that $M$ has rank and $M\otimes_RM^{\ast}$ is torsion-free. If $R$ has minimal multiplicity or $M$ is Ulrich, then $M$ is free. In fact, $R$ is regular when $M$ is Ulrich.
\end{cor}

\begin{proof} Note that $M\otimes_RM^{\ast}$ is maximal Cohen-Macaulay. If $M$ is Ulrich, then Corollary~\ref{corR} yields that $M^*$ is free. Since $\depth(R)=1$ and $M$ is torsion-free, we see that $M$ is free by \cite[2.13]{CG19}. 

Next, assume that $R$ has minimal multiplicity. Note that $M^* \cong \Omega^2_R(A) \oplus R^{\oplus n}$ for some $R$-module $A$ and $n \geq 0$. If $\Omega^2_R(A) = 0$, then $M^*$ is free and hence $M$ is free by \cite[2.13]{CG19}. If $\Omega^2_R(A) \neq 0$, then we can apply Corollary~\ref{corMM} to $M\otimes_R \Omega^2_R(A)$ to conclude that $M$ is free. 
\end{proof}

In the following observation $\torsion(M)$ denotes the torsion submodule of the module $M$. 

\begin{lem} \label{lemHW1} Let $R$ be a local ring and let $M$ be a Cohen-Macaulay $R$-module. If $M/\torsion(M)$ is a nonzero free $R$-module, then $M$ is free.
\end{lem}

\begin{proof} We have an exact sequence $0 \to \torsion(M) \to M \to R^{\oplus n} \to 0$ for some $n > 0$, which implies that $M \cong R^{\oplus n} \oplus  \torsion(M)$. Suppose $ \torsion(M) \neq 0$. Then it follows that \[\dim(R)-1\geq \dim_R(\torsion(M)) \geq \depth_R(\torsion(M)) \geq \depth_R(M)=\dim_R(M)=\dim(R),\] a contradiction. Consequently, $\torsion(M)=0$, and hence we have $M \cong R^{\oplus n}$.
\end{proof}

\begin{cor} \label{corHW1} Let $R$ be a local ring and let $M$ be an Ulrich $R$-module such that $M\otimes_RM^{\ast}$ is Cohen-Macaulay. Assume at least one of the following conditions holds:
\begin{enumerate}[\rm(i)]
\item $M$ is faithful.
\item $\Tor_1^R(M,M^{\ast})$ is torsion and $\Supp_R(M)=\Spec(R)$
\end{enumerate}
Then $M^{\ast}$ is free. Moreover, if $\depth(R)\leq 1$, then $M$ is free and $R$ is regular.
\end{cor}

\begin{proof} As $\dim(M\otimes_RM^{\ast}) = \dim(M)$, Theorem \ref{mainthm} implies that $M^{\ast}$ is free. Let $\overline{M}:=M/\torsion(M)$ denote the torsion-free part of $M$. Then it follows that $\overline{M}^{\ast} \cong M^{\ast}$, and hence $\overline{M}^{\ast}$ is free. If $\depth(R)\leq 1$, since $\overline{M}$ is torsion-free, it follows from \cite[2.13]{CG19} that $\overline{M}$ is free.  Now Lemma \ref{lemHW1} implies that $M$ is free. But as $M$ is Ulrich, $R$ must be regular.
\end{proof}

The following examples show that Corollary \ref{corR} is sharp:

\begin{eg} \label{egT1} Let $R=k[\![x,y]\!]/(xy)$ and $M=N=R/(x)$. Then $R$ is reduced, and $M$ and $M\otimes_RM$ are maximal Cohen-Macaulay Ulrich modules, but $M$ is not free. Note that $M$ does not have rank. Note also that $\Supp_R(M)=\{(x), (x,y)\}\neq \Spec(R)$. This example also shows that full support assumption in Theorem \ref{mainthm}(ii)(c) is needed: $\Tor_1^R(M,N)\cong k$ and hence it is torsion. 
\end{eg}

\begin{eg} (\cite[4.3]{HW94}) \label{egT2} Let $R=k[\![t^4, t^5, t^6]\!]$, $I=(t^4, t^5)$ and $J=(t^4, t^6)$. Then $R$ is a one-dimensional complete intersection domain and $I$, $J$ and $I\otimes_RJ$ are non-free maximal Cohen-Macaulay $R$-modules, each of which has rank. Note that $I$ and $J$ are not Ulrich modules as $\e_R(I)=\e_R(J)=\e(R)=4>\mu_R(I)=\mu_R(J)=2$.

Let $L=IJ$. Then, since $I\otimes_RJ$ is torsion-free, it follows that $IJ \cong I\otimes_RJ$. Therefore $L$ is an Ulrich module since $\mu_R(L)=4=\e_R(L)$. However $L \otimes_R L$ has torsion, that is, it is not maximal Cohen-Macaulay.
\end{eg}

\begin{eg}(\cite[7.3]{GT15}) \label{egT3} Let $R=k[\![t^9, t^{10}, t^{11}, t^{12}, t^{15}]\!]$, $I=(t^9, t^{10})$ and $J=(t^9, t^{12})$. Then $R$ is a one-dimensional domain which is not Gorenstein, and $I$, $J$ and $I\otimes_RJ \cong \omega_R$ are non-free maximal Cohen-Macaulay $R$-modules, each of which has rank. Note that $I$ and $J$ are not Ulrich modules as $\e_R(I)=\e_R(J)=\e(R)=9>\mu_R(I)=\mu_R(J)=2$.
\end{eg}

Finally, we examine the Ulrich properties when tensoring an Ulrich module with iterated tensor products or exterior powers. For this result, we recall that, if $M$ is an $R$-module, then the anti-symmetrization map $\iota_M^i:\bigwedge^i_R(M) \to \bigotimes^i_R M$ is given on elementary wedges by \[x_1 \wedge x_2 \cdots \wedge x_i \mapsto \sum_{\sigma \in S_i} (-1)^{\sgn(\sigma)} x_{\sigma(1)} \otimes_R x_{\sigma(2)} \otimes_R \cdots \otimes_R x_{\sigma(i)}.\] 
It is well known that $\iota^i_M$ is a split injection for any $i$ when $M$ is free. We provide a proof of this fact for lack of a reference in sufficient generality.

\begin{lem}\label{splitinjection} Let $R$ be a local ring. If $N$ is a free $R$-module, then $\iota_N^i$ is a split injection for each $i$. 
\end{lem}

\begin{proof} Let $e_1,\dots,e_n$ be a basis for $N$. We may assume $i \le n$. We first claim that $\iota^i_N$ is injective. Pick $r:=\sum_{j_1<\cdots <j_i} r_{j_1 \cdots j_i}(e_{j_1} \wedge \cdots \wedge e_{j_i}) \in \ker(\iota^i_N)$. Then we have \[\sum_{j_1<\cdots <j_i} \sum_{\sigma \in S_i} (-1)^{\sgn \sigma} r_{j_1 \cdots j_i}\big(e_{\sigma(j_1)} \otimes_R \cdots \otimes_R e_{\sigma(j_i)}\big),\] where the terms $e_{\sigma(j_1)} \otimes_R \cdots \otimes_R e_{\sigma(j_i)}$ that appear in this sum are all distinct basis elements of $\bigotimes^i_R N$. It follows that each $r_{j_1 \cdots j_i}=0$. Thus $r=0$, so $\iota^i_N$ is injective. 
We thus have a short exact sequence $0 \rightarrow \bigwedge^i_R N \xrightarrow{\iota^i_N} \bigotimes^i_R N \to C \to 0$. Applying $- \otimes_R k$, we obtain an exact sequence 
\[0 \rightarrow \Tor^R_1(C,k) \rightarrow \Big(\bigwedge^i_R N\Big) \otimes_R k \xrightarrow{\iota^i_N \otimes_R k} \Big(\bigotimes^i_R N\Big) \otimes_R k \to C \otimes_R k \to 0.\]
But there is a commutative diagram:
\[\begin{tikzcd}
	{\big(\bigwedge^i_R N\big) \otimes_R k} & {\big(\bigotimes^i_R N\big) \otimes_R k} \\
	{\bigwedge^i_k(N \otimes_R k)} & {\bigotimes^i_R (N \otimes_R k)}
	\arrow["{\iota^i_N \otimes k}", from=1-1, to=1-2]
	\arrow["{\iota^i_{N \otimes k}}", from=2-1, to=2-2]
	\arrow[from=1-1, to=2-1]
	\arrow[from=1-2, to=2-2]
\end{tikzcd}\]
whose vertical arrows are the natural isomorphisms. From above, we know that $\iota^i_{N \otimes_R k}$ is injective, so commutativity of the diagram implies that $\iota^i_N \otimes k$ is injective as well. Thus $\Tor^R_1(C,k)=0$, which implies that $C=\coker(\iota^i_N)$ is free over $R$. Thus $\iota^i_N$ is a split injection. 
\end{proof}

\begin{prop} \label{fitting} Let $R$ be a local ring and let $M, \,N_1, \ldots, N_t$ be nonzero $R$-modules. Assume:
\begin{enumerate}[\rm(i)]
\item $M$ is unmixed and $\e_R(M)=\mu_R(M)$.
\item $\dim_R(M)=\dim_R(M\otimes_RN_i)$ for all $i=1, \ldots, t$.
\item $M\otimes_R N_i$ is Cohen-Nacaulay for all $i=1, \ldots, t$.
\end{enumerate}
Then the following hold:
\begin{enumerate}[\rm(a)]
\item $M \otimes_R \big(\bigotimes^t_{i=1} R^{\oplus \mu_R(N_i)}\big) \cong M \otimes_R \big(\bigotimes^t_{i=1} N_i\big)$.
\item $M \otimes_R N_1 \otimes_R \cdots \otimes_R N_t$ and $M \otimes_R \big(\bigwedge^r_R N_i\big)$ are Ulrich for all $i=1, \ldots, t$ and all $r \geq 0$. 
\item If $\Tor_1^R(M, N_1 \otimes_R \cdots \otimes_R N_t)=0$, then $N_i$ is free for all $i=1, \ldots, t$.
\end{enumerate}
\end{prop}

\begin{proof} For each $i$, choose a surjection $p_i:R^{\oplus \mu_R(N_i)} \twoheadrightarrow N_i$. By Lemma~\ref{lem:c}(i), since $M$ is unmixed, we see $\ker(M \otimes p_i) = 0$. Thus $M \otimes p_i$ is an isomorphism, which implies that $M$ is Cohen-Macaulay and hence Ulrich. 
Inductively, we see that $M \otimes \big(\bigotimes_{i=1}^t p_i\big): M \otimes_R \big(\bigotimes^t_{i=1} R^{\oplus \mu_R(N_i)}\big) \to M \otimes_R \big(\bigotimes^t_{i=1} N_i\big)$ is an isomorphism, from which it follows that $M \otimes_R \big(\bigotimes^t_{i=1} N_i\big)$ is Ulrich. For the claim about modules $M \otimes_R \big(\bigwedge^r_R N_i\big)$, we set $N:=N_i$ and $p:=p_i$. There is a commutative diagram:  
\[\begin{tikzcd}
	{\bigwedge^r_R R^{\oplus \mu_R(N)}} & {\bigwedge^r_RN} \\
	{\bigotimes^r_R R^{\oplus \mu_R(N)}} & {\bigotimes^r_R N}
	\arrow["{\bigotimes^r p}", two heads, from=2-1, to=2-2]
	\arrow["{\bigwedge^r p}", two heads, from=1-1, to=1-2]
	\arrow["{\iota^r_{R^{\oplus \mu_R(N)}}}"', from=1-1, to=2-1]
	\arrow["{\iota^r_N}", from=1-2, to=2-2]
\end{tikzcd}\]
Applying $M \otimes_R -$ to this diagram, we obtain another commutative diagram:
\[\begin{tikzcd}
	{M \otimes_R \big(\bigwedge^r_R R^{\oplus \mu_R(N)}\big)} & {M \otimes_R \big(\bigwedge^r_R N\big)} \\
	{M \otimes_R \big(\bigotimes^r_R R^{\oplus \mu_R(N)}\big)} & {M \otimes_R \big(\bigotimes^r_R N\big)}
	\arrow["{M \otimes (\bigotimes^r p)}", from=2-1, to=2-2]
	\arrow["{M \otimes (\bigwedge^r p)}", from=1-1, to=1-2]
	\arrow["{M \otimes \iota^r_{R^{\oplus \mu_R(N)}}}"', from=1-1, to=2-1]
	\arrow["{M \otimes \iota^r_N}", from=1-2, to=2-2]
\end{tikzcd}\]
From Lemma \ref{splitinjection}, we have that $\iota^r_{R^{\oplus \mu_R(N)}}$ is a split injection,  so $M \otimes \iota^r_{R^{\oplus \mu_R(N)}}$ remains injective. Consequently, it follows from above that $M \otimes \big(\bigotimes^r p\big)$ is an isomorphism. The commutativity of the diagram then forces $M \otimes \big(\bigwedge^r p\big)$ to be injective, and hence it is an isomorphism. Thus $M \otimes_R \big(\bigwedge^r_R N\big)$ is Ulrich as well.

Next, assume $\Tor_1^R(M,N_1 \otimes_R \cdots \otimes_R N_t)=0$. Then we have the following short exact sequence:
$$0 \to M \otimes_R \Omega^1_R (N_1 \otimes_R \cdots \otimes_R N_t) \to M \otimes_R \Big(\bigotimes^t_{i=1} R^{\oplus \mu_R(N_i)} \Big) \xrightarrow{M \otimes (\bigotimes_{i=1}^t p_i)}  M \otimes_R \Big(\bigotimes^t_{i=1} N_i \Big).$$
As we know $M \otimes \big(\bigotimes_{i=1}^t p_i\big)$ is bijective, it follows that $M \otimes_R \Omega^1_R (N_1 \otimes_R \cdots \otimes_R N_t)=0$. As $M\neq 0$, this forces $\Omega^1_R (N_1 \otimes_R \cdots \otimes_R N_t)$ to be zero, which forces each $N_i$ to be free. 
\end{proof}

The next example shows the hypothesis that $M$ is Ulrich in Proposition \ref{fitting} cannot be relaxed to even assuming $M$ is $2$-Ulrich.

\begin{eg} Let $R=k[\![x,y,z]\!]/(xy)$ and let $M=N=R/(z)$. Then $M\otimes_R N \cong M$ is Cohen-Macaulay, but not Ulrich since $\mu_R(M) = 1 \neq 2 = \e_R(M)$.
\end{eg}

\section*{Acknowledgements}
Part of this work started when Araya and Yao visited West Virginia University (WVU) in August 2018; we thank Tokuji Araya for his comments and help on an earlier version of the manuscript. We also thank Jonathan Monta\~no for helpful discussions. Part of this work was completed during a visit of Lyle, Takahashi and Yao to WVU in May 2023. We thank the WVU School of Mathematical and Data Sciences for these opportunities. Takahashi was partly supported by JSPS Grant-in-Aid for Scientific Research 23K03070.

\bibliographystyle{plain}
\bibliography{mybib}

\begin{thebibliography}{10}

\bibitem{Av10}
Luchezar~L. Avramov.
\newblock Infinite free resolutions.
\newblock In {\em Six lectures on commutative algebra}, Mod. Birkh\"{a}user
  Class., pages 1--118. Birkh\"{a}user Verlag, Basel, 2010.

\bibitem{BH89}
J\"orgen Backelin and J\"urgen Herzog.
\newblock On {U}lrich-modules over hypersurface rings.
\newblock In {\em Commutative algebra ({B}erkeley, {CA}, 1987)}, volume~15 of
  {\em Math. Sci. Res. Inst. Publ.}, pages 63--68. Springer, New York, 1989.

\bibitem{BH87}
Joseph~P. Brennan, J\"urgen Herzog, and Bernd Ulrich.
\newblock Maximally generated {C}ohen-{M}acaulay modules.
\newblock {\em Math. Scand.}, 61(2):181--203, 1987.

\bibitem{BH93}
Winfried Bruns and J\"{u}rgen Herzog.
\newblock {\em Cohen-{M}acaulay rings}, volume~39 of {\em Cambridge Studies in
  Advanced Mathematics}.
\newblock Cambridge University Press, Cambridge, 1993.

\bibitem{CD14}
Olgur Celikbas, Hailong Dao, and Ryo Takahashi.
\newblock Modules that detect finite homological dimensions.
\newblock {\em Kyoto J. Math.}, 54(2):295--310, 2014.

\bibitem{CG19}
Olgur Celikbas, Shiro Goto, Ryo Takahashi, and Naoki Taniguchi.
\newblock On the {I}deal {C}ase of a {C}onjecture of {H}uneke and {W}iegand.
\newblock {\em Proc. Edinb. Math. Soc. (2)}, 62(3):847--859, 2019.

\bibitem{Co95}
Petra Constapel.
\newblock Vanishing of {T}or and torsion in tensor products.
\newblock {\em Comm. Algebra}, 24(3):833--846, 1996.

\bibitem{CM15}
Laura Costa and Rosa~M. Mir\'{o}-Roig.
\newblock {$GL(V)$}-invariant {U}lrich bundles on {G}rassmannians.
\newblock {\em Math. Ann.}, 361(1-2):443--457, 2015.

\bibitem{CM12}
Laura Costa, Rosa~M. Mir\'o-Roig, and Joan Pons-Llopis.
\newblock The representation type of {S}egre varieties.
\newblock {\em Adv. Math.}, 230(4-6):1995--2013, 2012.

\bibitem{DG23}
Souvik Dey and Dipankar Ghosh.
\newblock Complexity and rigidity of {U}lrich modules, and some applications.
\newblock {\em Math. Scand.}, 129(2):209--237, 2023.

\bibitem{DK23}
Souvik Dey and Toshinori Kobayashi.
\newblock Vanishing of (co)homology of {B}urch and related submodules.
\newblock {\em Illinois J. Math.}, 67(1):101--151, 2023.

\bibitem{ES03}
David Eisenbud, Frank-Olaf Schreyer, and Jerzy Weyman.
\newblock Resultants and {C}how forms via exterior syzygies.
\newblock {\em J. Amer. Math. Soc.}, 16(3):537--579, 2003.

\bibitem{GO16}
Shiro Goto, Kazuho Ozeki, Ryo Takahashi, Kei-Ichi Watanabe, and Ken-Ichi
  Yoshida.
\newblock Ulrich ideals and modules over two-dimensional rational
  singularities.
\newblock {\em Nagoya Math. J.}, 221(1):69--110, 2016.

\bibitem{GT15}
Shiro Goto, Ryo Takahashi, Naoki Taniguchi, and Hoang Le~Truong.
\newblock Huneke-{W}iegand conjecture and change of rings.
\newblock {\em J. Algebra}, 422:33--52, 2015.

\bibitem{Ha05}
Douglas Hanes.
\newblock On the {C}ohen-{M}acaulay modules of graded subrings.
\newblock {\em Trans. Amer. Math. Soc.}, 357(2):735--756, 2005.

\bibitem{HH05}
Douglas Hanes and Craig Huneke.
\newblock Some criteria for the {G}orenstein property.
\newblock {\em J. Pure Appl. Algebra}, 201(1-3):4--16, 2005.

\bibitem{Ha99}
Douglas~Allen Hanes.
\newblock {\em Special conditions on maximal {C}ohen-{M}acaulay modules, and
  applications to the theory of multiplicities}.
\newblock ProQuest LLC, Ann Arbor, MI, 1999.
\newblock Thesis (Ph.D.)--University of Michigan.

\bibitem{He70}
J\"{u}rgen Herzog.
\newblock Generators and relations of abelian semigroups and semigroup rings.
\newblock {\em Manuscripta Math.}, 3:175--193, 1970.

\bibitem{HU91}
J\"{u}rgen Herzog, Bernd Ulrich, and J\"{o}rgen Backelin.
\newblock Linear maximal {C}ohen-{M}acaulay modules over strict complete
  intersections.
\newblock {\em J. Pure Appl. Algebra}, 71(2-3):187--202, 1991.

\bibitem{HI19}
Craig Huneke, Srikanth~B. Iyengar, and Roger Wiegand.
\newblock Rigid ideals in {G}orenstein rings of dimension one.
\newblock {\em Acta Math. Vietnam.}, 44(1):31--49, 2019.

\bibitem{HW94}
Craig Huneke and Roger Wiegand.
\newblock Tensor products of modules and the rigidity of {${\rm Tor}$}.
\newblock {\em Math. Ann.}, 299(3):449--476, 1994.

\bibitem{JL07}
David~A. Jorgensen and Graham~J. Leuschke.
\newblock On the growth of the {B}etti sequence of the canonical module.
\newblock {\em Math. Z.}, 256(3):647--659, 2007.

\bibitem{JL07-E}
David~A. Jorgensen and Graham~J. Leuschke.
\newblock Erratum: ``{O}n the growth of the {B}etti sequence of the canonical
  module'' [{M}ath. {Z}. {\bf 256} (2007), no. 3, 647--659; mr2299575].
\newblock {\em Math. Z.}, 260(3):713--715, 2008.

\bibitem{KT19}
Toshinori Kobayashi and Ryo Takahashi.
\newblock Ulrich modules over {C}ohen--{M}acaulay local rings with minimal
  multiplicity.
\newblock {\em Q. J. Math.}, 70(2):487--507, 2019.

\bibitem{LW12}
Graham~J. Leuschke and Roger Wiegand.
\newblock {\em Cohen-{M}acaulay representations}, volume 181 of {\em
  Mathematical Surveys and Monographs}.
\newblock American Mathematical Society, Providence, RI, 2012.

\bibitem{LM19}
Justin Lyle and Jonathan Monta\~{n}o.
\newblock Extremal growth of {B}etti numbers and trivial vanishing of
  (co)homology.
\newblock {\em Trans. Amer. Math. Soc.}, 373(11):7937--7958, 2020.

\bibitem{Ma20}
Linquan Ma.
\newblock Lim {U}lrich sequences and {L}ech's conjecture.
\newblock {\em Invent. Math.}, 231(1):407--429, 2023.

\bibitem{Ma00}
Vladimir Ma\c{s}ek.
\newblock Gorenstein dimension and torsion of modules over commutative
  {N}oetherian rings.
\newblock In {\em Communications in Algebra}, volume~28, pages 5783--5811.
  2000.
\newblock Special issue in honor of Robin Hartshorne.

\bibitem{SV74}
Judith~D. Sally and Wolmer~V. Vasconcelos.
\newblock Stable rings.
\newblock {\em J. Pure Appl. Algebra}, 4:319--336, 1974.

\bibitem{Ul84}
Bernd Ulrich.
\newblock Gorenstein rings and modules with high numbers of generators.
\newblock {\em Math. Z.}, 188(1):23--32, 1984.

\bibitem{Yh21}
Farrah~C. {Yhee}.
\newblock Ulrich modules and weakly lim ulrich sequences do not always exist.
\newblock arXiv:2104.05766, 2021.

\end{thebibliography}

\end{document}